\documentclass{article}
                                             
\usepackage{cite}      
\usepackage{graphicx}
\usepackage{float}
\usepackage{subfigure} 
\usepackage{amsmath}
\usepackage{amsthm}
\usepackage{amssymb}
\usepackage{enumerate}
\usepackage{deleq}   
\usepackage{type1cm} 
\usepackage{color}

\RequirePackage[left=1in,right=1in,top=1in,bottom=1in]{geometry}

\newtheorem{defnt}{Definition}
\newtheorem{remark}{Remark}
\newtheorem{thm}{Theorem}
\newtheorem{lemma}{Lemma}

\newtheorem{prop}{Proposition}

\DeclareMathOperator*{\argmin}{arg\,min}
\DeclareMathOperator*{\argmax}{arg\,max}

\title{An Input-Output Construction of Finite State $\rho / \mu$ Approximations for Control Design}
\author{Danielle C.~Tarraf\footnote{The author is
with the Department of Electrical \& Computer Engineering Department at The Johns Hopkins University, Baltimore, MD, 21218  
(dtarraf@jhu.edu).}}
\date{}

\begin{document}

\maketitle
                    
\begin{abstract}
We consider discrete-time plants that interact with their controllers via fixed discrete alphabets.
For this class of systems,
and in the absence of exogenous inputs,
we propose a general, conceptual procedure for constructing a sequence of finite state approximate models 
starting from finite length sequences of input and output signal pairs.
We explicitly derive conditions under which the proposed construct,
used in conjunction with a particular generalized structure,
satisfies desirable properties of $\rho/\mu$ approximations thereby
leading to nominal deterministic finite state machine models that can be used in certified-by-design controller synthesis.
We also show that the cardinality of the minimal disturbance alphabet 
that can be used in this setting equals that of the sensor output alphabet.
Finally, we show that the proposed construct satisfies a relevant semi-completeness property.
\end{abstract}

\section{Introduction}
\label{Sec:Introduction}

\subsection{Motivation}
\label{SSec:Motivation}

Cyber-physical systems, involving tightly integrated physical and computational components,
are omni-present in modern engineered systems.
These systems are fundamentally complex, and pose multiple challenges to the control engineer \cite{TR:Lee2008}.
In order to effectively address these challenges, 
there is an inevitable need to move to abstractions 
or model reduction schemes that can handle dynamics 
and computation in a unified framework.
Ideally, 
an abstraction or  model complexity reduction approach should provide a 
lower complexity model that is more easily amenable to analysis, synthesis and optimization,
as well as a rigorously quantifiable assessment of the quality of approximation.
This would allow one to certify the performance of a controller designed for the lower complexity model 
and implemented in the actual system faithfully captured by the original model,
without the need for extensive simulation or testing.

The problem of approximating systems involving dynamics and computation 
(cyber-physical systems)
or discrete and analog effects (hybrid systems)
by simpler systems has been receiving much attention over the past two decades 
\cite{JOUR:AlHeLaPa2000, CHAPTER:TiKh2002}. 
In particular, 
the problem of constructing {\it finite} state approximations of hybrid systems has 
been the object of intense study, 
due to the rampant use of finite state machines as models of computation or software,
as well as their amenability to tractable 
analysis \cite{CONF:TaDaMe2005} and 
control synthesis \cite{BOOK:MatSav2000, JOUR:KoImHi2011}
(though tractable does not always mean computationally efficient!).

\subsection{Overview of the Contribution}
\label{SSec: Contribution}

In a previous effort \cite{JOUR:Tarraf2012}, we proposed a notion of finite state 
approximation for `systems over finite alphabets', 
basically plants that are constrained to interact with their feedback controllers 
by sending and receiving signals taking their values in fixed, finite alphabet sets.
We refer to this notion of approximation as a `$\rho/\mu$ approximation', 
to highlight the fact that is is compatible with the analysis \cite{JOUR:TaMeDa2008} 
and synthesis \cite{JOUR:TaMeDa2011} tools we had previously developed 
for systems whose properties and/or performance objectives are described in terms of 
$\rho/\mu$ gain conditions.
Note that the proposed notion of $\rho/\mu$ approximation explicitly identified
those properties that the approximate models need to satisfy in order to enable
certified-by-design controller synthesis.
However, it did not restrict us to a particular constructive algorithm
for generating these approximations.

In this paper, we propose and analyze a new\footnote{Early versions of this construct and its
analysis were presented in \cite{CONF:Tarraf2012,CONF:Tarraf2012a, CONF:Tarraf2013b}
An implementation of this construct demonstrating its application to a specific example was presented in \cite{CONF:AalTar2012}.}
approach for generating $\rho/\mu$ approximations of a given plant and performance objective.
In contrast to the state-space based construction
presented as a simple illustrative example in \cite{JOUR:Tarraf2012}, 
which was specifically tailored to the dynamics in question,
the present construct is a general methodology that is applicable to 
arbitrary plants over finite alphabets provided that: 
(i) They are not subject to exogenous inputs,
and (ii) their outputs are a function of the state only 
(i.e. analogous to strictly proper transfer functions in the LTI setting).

Our construct essentially associates states of the approximate model with 
finite length subsequences of input-output pairs of the plant.
Since the underlying alphabets are finite, 
the set of possible input-output pairs of a given length is also finite.
The resulting approximate models thus have finite state-space,
and are shown to satisfy desirable properties of $\rho/\mu$ approximations under some clearly identified conditions, 
thereby rendering them useable for control synthesis.
Our construct is conceptual,
in the sense that we do not address computational issues 
that may arise due to the complexity of the underlying dynamics.
As such, our contribution is a general methodology, as opposed to a computational framework, 
for generating finite state $\rho/\mu$ approximations,
and a rigorous analysis of the properties of this construct.

\subsection{Related Work}
\label{SSec:RelatedWork}

Automata and finite state models have been previously employed
as abstractions or approximate models of more complex dynamics for the purpose of control design.
We survey the directions most relevant to our work in what follows.

One research direction makes use of non-deterministic finite state automata 
constructed so that their input/output behavior contains that of the original model
(these approximations are sometimes referred to as `qualitative models') 
\cite{JOUR:Lu1994,JOUR:RaOY1998,JOUR:LuNiSc1999}. 
Controller synthesis can then be formulated as a supervisory control problem, 
addressed using the Ramadge-Wonham framework \cite{JOUR:RaWo1987, JOUR:RaWo1989}.
More recently, progress has been made in reframing these results \cite{JOUR:MoRa1999, JOUR:MoRaOY2002}
in the context of Willems' behavioral theory and $l$-complete systems \cite{MAGAZINE:Wi2007}.
Our construct bears some resemblance to algorithms employed
in constructing qualitative models.
However, our notion of $\rho/\mu$ approximation is fundamentally different
from the notion of qualitative models, as it seeks to explicitly 
quantify the approximation error in the spirit of robust control.

A second research direction,
influenced by the theory of bisimulation in concurrent processes \cite{CHAPTER:Pa1981,BOOK:Mi1989},
makes use of bisimulation and simulation abstractions of the original plant.
These approaches,
which typically address full state feedback problems, effectively ensure that
the set of state trajectories of the original model is exactly matched by (bisimulation),
contained in (simulation),
matched to within some distance $\epsilon$ by (approximate bisimulation),
or contained to within some distance $\epsilon$ in (approximate simulation),
the set of state trajectories of the finite state abstraction 
\cite{JOUR:GiPa2007, JOUR:Ta2008, JOUR:TaAmJuPa2008, JOUR:PoGiTa2008}.
The performance objectives are typically formulated as constraints
on the state trajectories of the original hybrid system,
and controller synthesis is a two step procedure: 
A finite state supervisory controller is first designed, and subsequently refined 
to yield a certified-by-design hybrid controller for the original plant \cite{BOOK:Ta2009}.

Other related research directions make use of 
symbolic models \cite{JOUR:KloBel2008, MAGAZINE:BeBiEgFrKlPa2007} ,
approximating automata \cite{JOUR:CuKrNi1998,JOUR:ChuKro2000, JOUR:Reiszi2011},
and finite quotients of the system \cite{JOUR:ChuKro2001, JOUR:YorBel2010}.
While the subject of input-output robustness of discrete systems
has been garnering more attention recently \cite{CONF:TaBaCSM2012},
we are not aware of any alternative notions of discrete approximation 
developed in conjunction with that work.

Of course, the idea of using finite length sequences of inputs and outputs 
is widely employed in system identification \cite{BOOK:Lj1999}. 
However, the setup of interest to us is fundamentally different for three reasons:
First, the dynamics of the plant are exactly known.
Second, the data can be generated in its entirety.
Third, the data is exact and uncorrupted by noise.

Finally, the present construct differs from our first effort reported in \cite{CONF:TaDu2011}, 
as it approximates the performance objectives as well as the dynamics of the systems,
and moreover leads to a finite state nominal model with deterministic transitions.

\subsection{Organization and Notation}
\label{SSec:OrganizationNotation}

We begin in Section \ref{Sec:Preliminaries} by reviewing the relevant notion of 
$\rho/\mu$ approximation as well as basic concepts that will be useful in our development.
We state the problem of interest in Section \ref{Sec:ProblemSetup}. 
We revisit a special structure in Section \ref{Sec:SpecialStructure}:
We demonstrate its relevance to $\rho/\mu$ approximations,
and we address the related question of disturbance alphabet choice.
We present our construct in Section \ref{Sec:Construction}
and give the intuition behind it.
We show that the resulting approximate models satisfy several of the desired 
$\rho/\mu$ approximation properties in Section \ref{Sec:Properties},
and we address the question of ensuring finiteness of the approximation
error gain.
We demonstrate further relevant properties in Section \ref{Sec:Completeness},
highlighting the completeness of this construct.
We conclude with directions for future work in Section \ref{Sec:Conclusions}.

We employ fairly standard notation:
$\mathbb{Z}_+$ and $\mathbb{R}_+$ denote the non-negative integers and non-negative reals, respectively. 
Given a set $\mathcal{A}$, 
$\mathcal{A}^{\mathbb{Z}_+}$ and $2^{\mathcal{A}}$
denote the set of all infinite sequences over $\mathcal{A}$ 
(indexed by $\mathbb{Z}_+$) 
and the power set of $\mathcal{A}$, respectively. 
The cardinality of a (finite) set $\mathcal{A}$ is denoted by $|\mathcal{A}|$.
Elements of $\mathcal{A}$ and $\mathcal{A}^{\mathbb{Z}_+}$ are denoted by $a$ 
and (boldface) $\mathbf{a}$, respectively. 
For $\mathbf{a} \in \mathcal{A}^{\mathbb{Z}_+}$, $a(i)$ denotes its $i^{th}$ term.
For $f: A \rightarrow B$, $C \subset B$, $f^{-1}(C) = \{ a \in A | f(a) \in C \}$. 
For $f: A \rightarrow B$ and $g: B \rightarrow C$, $g \circ f$ denotes the composition of $f$ and $g$,
that is the function $g \circ f : A \rightarrow C$ defined by $g \circ f(a) = g(f(a))$.
Given $P \subset (\mathcal{U} \times \mathcal{R})^{\mathbb{Z}_+} \times (\mathcal{Y} \times \mathcal{V})^{\mathbb{Z}_+}$
and a choice $\mathbf{u_o} \in \mathcal{U}^{\mathbb{Z}_+}$,
$\mathbf{y_o} \in \mathcal{Y}^{\mathbb{Z}_+}$,
$P|_{\mathbf{u_o,y_o}}$ denotes the (possibly empty) subset of $P$ defined as 
$P |_{\mathbf{u_o},\mathbf{y_o}} = \Big\{ \Big( (\mathbf{u},\mathbf{r}),(\mathbf{y},\mathbf{v}) \Big) \in P \Big| \mathbf{u}=\mathbf{u_o} \textrm{   and   } \mathbf{y}=\mathbf{y_o} \Big\}$.

\section{Preliminaries}
\label{Sec:Preliminaries}

In our development, it is often convenient to view a discrete-time dynamical system as a set of 
feasible signals, even when a state-space description of the system is available.
We thus begin this section by briefly reviewing this `feasible signals' view of systems. 
We then present the recently proposed notion of $\rho/\mu$ approximation 
specialized to the class of systems of interest (namely systems with no exogenous inputs),
and we state the relevant control synthesis result.

\subsection{Systems and Performance Specifications}
\label{SSec:SystemsSpecs}

Readers are referred to \cite{JOUR:TaMeDa2008} for a more detailed treatment
of the basic concepts reviewed in this section.
A discrete-time signal is an infinite sequence over some prescribed set (or `alphabet').

\begin{defnt}
\label{def:system}
A discrete-time system $S$ is a set of pairs of signals, $S \subset \mathcal{U}^{\mathbb{Z}_+} \times \mathcal{Y}^{\mathbb{Z}_+}$,
where $\mathcal{U}$ and $\mathcal{Y}$ are given alphabets.
\end{defnt}

A discrete-time system is thus a process characterized by its feasible signals set. 
This description can be considered an extension of the 
graph theoretic approach \cite{JOUR:GeSm1997} to the finite alphabet setting,
and also shares some similarities with the behavioral approach \cite{MAGAZINE:Wi2007}
though we insist on differentiating between input and output signals upfront.
In this setting, 
system properties of interest are captured by means of integral `$\rho/\mu$ constraints' on the feasible signals.

\begin{defnt} 
\label{def:GainStability}
Consider a system $S \subset \mathcal{U}^{\mathbb{Z}_+} \times \mathcal{Y}^{\mathbb{Z}_+}$ and let $\rho: \mathcal{U} \rightarrow \mathbb{R}$ 
and $\mu: \mathcal{Y} \rightarrow \mathbb{R}$ be given functions. $S$ is \textit{$\rho / \mu$ stable} if there exists a finite 
non-negative constant $\gamma$ such that 
\begin{equation}
\label{eq:gain}
\inf_{T \geq 0} \sum_{t=0}^{T} \gamma \rho (u(t)) - \mu (y(t))  > - \infty.
\end{equation}
is satisfied for all $(\mathbf{u},\mathbf{y})$ in $S$.
\end{defnt}

In particular, when $\rho$, $\mu$ are non-negative (and not identically zero), a notion of `gain' can be defined.

\begin{defnt}
\label{def:Gain}
Consider a system $S \subset \mathcal{U}^{\mathbb{Z}_+} \times \mathcal{Y}^{\mathbb{Z}_+}$.
Assume that $S$ is $\rho / \mu$ stable for $\rho: \mathcal{U} \rightarrow \mathbb{R}_+$ and
$\mu: \mathcal{Y} \rightarrow \mathbb{R}_+$, and that neither function is identically zero.
The $\rho / \mu$ \textit{gain of} $S$ is the infimum of $\gamma$ such that (\ref{eq:gain}) is satisfied.
\end{defnt}
 
Note that these notions of `gain stability' and `gain' can be considered extensions of the 
classical definitions to the finite alphabet setting. 
In particular, when $\mathcal{U}$, $\mathcal{Y}$ are Euclidean vector spaces and
$\rho$, $\mu$ are Euclidean norms, 
we recover $l_2$ stability and $l_2$ gain.
We are specifically interested in discrete-time plants that interact with 
their controllers through fixed discrete alphabets in a setting where no exogenous 
input is present:

\begin{defnt}
\label{def:system}
A \textit{system over finite alphabets} $S$ is a discrete-time system
$S \subset \mathcal{U}^{\mathbb{Z}_+} \times (\mathcal{Y} \times \mathcal{V})^{\mathbb{Z}_+}$
whose alphabets $\mathcal{U}$ and $\mathcal{Y}$ are finite.
\end{defnt}

Here $\mathbf{u} \in \mathcal{U}^{\mathbb{Z}_+}$ 
represents the control input to the plant
while $\mathbf{y} \in \mathcal{Y}^{\mathbb{Z}_+}$ and $\mathbf{v} \in \mathcal{V}^{\mathbb{Z}_+}$
represent the sensor and performance outputs of the plant, respectively.
The plant dynamics may be analog, discrete or hybrid.
Alphabet $\mathcal{V}$ may be finite, countable or infinite.
The approximate models of the plant will be drawn from a specific class of models,
namely deterministic finite state machines:

\begin{defnt}
\label{def:DFM}
A deterministic finite state machine (DFM) is a discrete-time system
$S \subset \mathcal{U}^{\mathbb{Z}_+} \times \mathcal{Y}^{\mathbb{Z}_+}$,
with finite alphabets $\mathcal{U}$ and $\mathcal{Y}$,
whose feasible input and output signals $(\mathbf{u},\mathbf{y}) \in S$ are related by
\begin{eqnarray*}
q(t+1) & = & f (q(t), u(t)) \\
y(t) & = & g(q(t),u(t))
\end{eqnarray*}
where $t \in \mathbb{Z}_+$, $q(t) \in \mathcal{Q}$ for some finite set $\mathcal{Q}$ and 
some functions $f: \mathcal{Q} \times \mathcal{U} \rightarrow \mathcal{Q}$ 
and $g: \mathcal{Q} \times \mathcal{U} \rightarrow \mathcal{Y}$.
\end{defnt}

$\mathcal{Q}$, $f$ and $g$ are understood to represent the set of states of the DFM, 
its state transition map, and its output map, respectively, in the traditional state-space sense.
We single out deterministic finite state machines in which there is no direct feedthrough 
from particular inputs to particular outputs:

\begin{defnt}
A DFM $S \subset (\mathcal{U}_1 \times \hdots \times \mathcal{U}_{n_{I}})^{\mathbb{Z}_+} 
\times 
(\mathcal{Y}_1 \times \hdots \times \mathcal{Y}_{n_{O}})^{\mathbb{Z}_+}$ is 
$\mathcal{U}_i/\mathcal{Y}_j$ strictly proper if its $j^{th}$ output map is of the form
\begin{displaymath}
y_{j} = g_{j} (q(t), u_{1}(t), \hdots, u_{i-1}(t), u_{i+1}(t), \hdots, u_{n_{I}}(t)),
\end{displaymath}
and strictly proper if it is $\mathcal{U}_i/\mathcal{Y}_j$ strictly proper for all
$i \in \{1,\hdots, n_I\}$ and $j \in \{1,\hdots, n_O \}$.
\end{defnt}

Finally, we introduce the following notation for convenience:
Given a system $P \subset \mathcal{U}^{\mathbb{Z}_+} \times (\mathcal{Y} \times \mathcal{V})^{\mathbb{Z}_+}$
and a choice of signals $\mathbf{u_o} \in \mathcal{U}^{\mathbb{Z}_+}$ and
$\mathbf{y_o} \in \mathcal{Y}^{\mathbb{Z}_+}$,
$P|_{\mathbf{u_o,y_o}}$ denotes the subset of feasible signals of $P$ 
whose first component is $\mathbf{u_o}$ and whose second component is $\mathbf{y_o}$.
That is
\begin{displaymath}
P |_{\mathbf{u_o},\mathbf{y_o}} = \Big\{ \Big( \mathbf{u},(\mathbf{y},\mathbf{v}) \Big) \in P \Big| \mathbf{u}=\mathbf{u_o} \textrm{   and   } \mathbf{y}=\mathbf{y_o} \Big\}.
\end{displaymath}
Note that $P|_{\mathbf{u_o,y_o}}$ may be an empty set for specific choices of $\mathbf{u_o}$ and $\mathbf{y_o}$.

\subsection{$\rho/ \mu$ Approximations for Control Synthesis}
\label{SSec:Approximation}

The following definition is adapted from \cite{JOUR:Tarraf2012} for the case where the plant is not
subject to exogenous inputs, of interest in this paper.
Note that in the absence of exogenous input, function $\rho$ drops out of the definition.
Nonetheless, we will continue to call this a ``$\rho/\mu$ approximation"
in keeping with the previously established terminology.

   \begin{figure*}[thpb]
       \centering
       \includegraphics[scale=0.4]{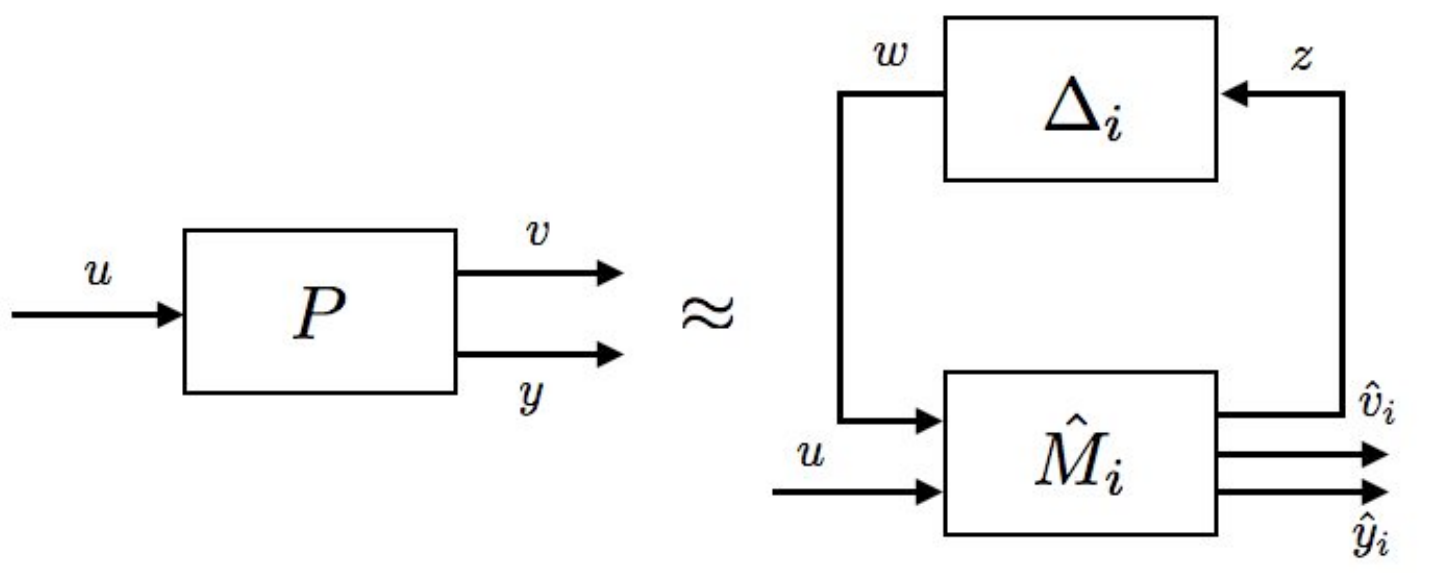}
       \caption{A finite state approximation of $P$}
       \label{Fig:Approximation}
   \end{figure*}

\begin{defnt}(Adapted from Definition 6 in \cite{JOUR:Tarraf2012})
\label{Def:DFMApproximation}
Consider a system over finite alphabets $P \subset \mathcal{U}^{\mathbb{Z}_+} \times (\mathcal{Y} \times \mathcal{V})^{\mathbb{Z}_+}$
and a desired closed loop performance objective
\begin{eqnarray}
\nonumber
\inf_{T \geq 0} \sum_{t=0}^{T} - \mu(v(t)) > -\infty \Leftrightarrow \\
\label{eq:ObjectiveP}
\sup_{T \geq 0} \sum_{t=0}^{T} \mu(v(t)) < \infty
\end{eqnarray}
for given function $\mu: \mathcal{V} \rightarrow \mathbb{R}$.
A sequence $\{\hat{M}_i\}_{i=1}^{\infty}$ of deterministic finite state machines
$\hat{M}_i \subset (\mathcal{U} \times \mathcal{W})^{\mathbb{Z}_+} 
\times (\mathcal{Y} \times \hat{\mathcal{V}}_i \times \mathcal{Z})^{\mathbb{Z}_+}$
with $\hat{\mathcal{V}}_i \subset \mathcal{V}$ is a {\boldmath $\rho / \mu$} 
\textbf{approximation} of $P$
if there exists a corresponding sequence of systems $\{\Delta_i\}_{i=1}^{\infty}$,
 $\Delta_i \subset \mathcal{Z}^{\mathbb{Z}^+} \times \mathcal{W}^{\mathbb{Z}_+}$, 
and non-zero functions $\rho _{\Delta}:\mathcal{Z}^{\mathbb{Z}_+} \rightarrow \mathbb{R}_+$,
$\mu_{\Delta}: \mathcal{W}^{\mathbb{Z}_+} \rightarrow \mathbb{R}_+$, such that for every $i$:

\begin{enumerate}[a)]

\item There exists a surjective map $\psi_i: P \rightarrow \hat{P}_i$
satisfying
\begin{equation}
\label{Eq:MapCondition}
\psi_i \Big( P|_{\mathbf{u},\mathbf{y}} \Big) \subseteq \hat{P}_i |_{\mathbf{u},\mathbf{y}}
\end{equation}
for all $(\mathbf{u},\mathbf{y}) \in \mathcal{U}^{\mathbb{Z_+}} \times \mathcal{Y}^{\mathbb{Z}_+}$,
where 
$\hat{P}_i \subset \mathcal{U}^{\mathbb{Z}_+} \times (\mathcal{Y} \times \hat{\mathcal{V}}_i)^{\mathbb{Z}_+} $ 
is the feedback interconnection of $\hat{M}_i$ and $\Delta_i$ as shown in Figure \ref{Fig:Approximation}.

\item For every feasible signal
$(\mathbf{u},(\mathbf{y},\mathbf{v})) \in P$, we have
\begin{equation}
\label{eq:Objectivenounds}
\mu(v(t)) \leq \mu(\hat{v}_{i+1}(t)) \leq \mu(\hat{v}_i(t)),  
\end{equation}
for all $t \in \mathbb{Z}_+$, where 
$$ (\mathbf{u},(\mathbf{\hat{y}_i},\mathbf{\hat{v}_i})) = \psi_i \Big( ( \mathbf{u},(\mathbf{y},\mathbf{v}))  \Big),$$
$$(\mathbf{u},(\mathbf{\hat{y}_{i+1}},\mathbf{\hat{v}_{i+1}})) = \psi_{i+1} \Big( (\mathbf{u},(\mathbf{y},\mathbf{v}))  \Big).$$

\item $\Delta_i$ is $\rho_{\Delta} / \mu_{\Delta}$ gain stable,
and moreover, the corresponding $\rho_{\Delta} / \mu_{\Delta}$ gains satisfy $\gamma_{i} \geq \gamma_{i+1}$. 

\end{enumerate}  
\end{defnt}
\smallskip

\begin{remark}
Intuitively,
the quality of the $i^{th}$ approximation is captured by the gain $\gamma_i$
of the approximation error system $\Delta_i$ (in condition c)),
and the gap between the original and auxiliary performance objectives (the outer inequality in condition b)).
We do not require strict inequalities in conditions b) and c),
to allow for instances where the sequence of approximate models recovers the original plant exactly 
after a finite number of steps (i.e. for some finite value of $i$),
or alternatively,
instances where it may not converge\footnote{Indeed, 
it is not clear to us that every system should admit an \emph{arbitrarily close finite state} approximation!} at all, 
but nonetheless provides a good enough approximation for the control problem at hand.
\end{remark}

Next, we review a result demonstrating that a $\rho/\mu$ approximation of the plant together with a new,
appropriately defined performance objective
may be used to synthesize certified-by-design controllers for the original plant and performance objective:

\begin{thm} (Adapted from Theorems 1 and 3 in \cite{JOUR:Tarraf2012})
Consider a plant $P$ and a $\rho/\mu$ approximation $\{\hat{M}_i\}_{i=1}^{\infty}$ as in Definition \ref{Def:DFMApproximation}.
If for some index $i$, there exists a controller $K \subset \mathcal{Y}^{\mathbb{Z}_+} \times \mathcal{U}^{\mathbb{Z}_+}$
such that the feedback interconnection of $\hat{M}_i$ and $K$,
$(\hat{M}_i,K) \subset \mathcal{W}^{\mathbb{Z}_+} \times (\hat{\mathcal{V}}_i \times \mathcal{Z})^{\mathbb{Z}_+}$,
satisfies
\begin{equation}
\label{eq:ObjectiveM}
\inf_{T \geq 0} \sum_{t=0}^{T} \tau \mu_{\Delta}(w(t)) - \mu(\hat{v}(t)) - \tau \gamma_i \rho_{\Delta}(z(t)) > -\infty
\end{equation}
for some $\tau >0$,
then the feedback interconnection of $P$ and $K$,
$(P,K) \subset \mathcal{V}^{\mathbb{Z}_+}$,
satisfies (\ref{eq:ObjectiveP}).
\end{thm}

\begin{remark}
In practice, the entire sequence of approximations is not constructed upfront:
Rather, the first element is constructed and control synthesis is attempted.
If synthesis fails, the next element of the sequence is constructed,
and so the process continues.
\end{remark}

Finally, synthesizing a {\it full state} feedback controller for a given DFM in order to satisfy given 
performance objectives of the form (\ref{eq:ObjectiveM}), 
for a {\it given} value of $\tau >0$, is a readily solvable problem:

\begin{thm} (Adapted from Theorem 4 in \cite{JOUR:TaMeDa2011})
\label{Thm:DPsynthesis}
Consider a DFM $M$ with state transition equation
$$ q(t+1)=f(q(t),u(t),w(t)),$$
and let $\sigma : \mathcal{Q} \times \mathcal{U} \times \mathcal{W} \rightarrow \mathbb{R}$ be given.
There exists a $\varphi:\mathcal{Q} \rightarrow \mathcal{U}$ such that the closed loop system $(M,\varphi)$ satisfies
\begin{equation}
\label{eq:SigmaObjective}
\inf_{T \geq 0} \sum_{t=0}^{T} \sigma(q(t),\varphi(q(t)),w(t)) > -\infty.
\end{equation}
iff 
the sequence of functions $J_k : \mathcal{Q} \rightarrow \mathbb{R}$, $k \in \mathbb{Z}_+$, defined recursively by
\begin{eqnarray}
\label{eq:Jsequence}
J_0 & = & 0 \\ 
\nonumber
J_{k+1} & = & \max \{0,\mathbb{T}(J_k)\}  
\end{eqnarray}
where 
$\displaystyle \mathbb{T}(J(q)) = \min_{u \in \mathcal{U}} \max_{w \in \mathcal{W}} \{-\sigma(q,u,w) + J(f(q,u,w)) \}$,
converges.
\end{thm}

Note that in particular, a gain condition such as (\ref{eq:ObjectiveM}), can be written in the form 
(\ref{eq:SigmaObjective}) as the outputs $\hat{v}$ and $z$ of $\hat{M}_i$ 
are functions of the state of $\hat{M}_i$ and its inputs.

\section{Problem Setup}
\label{Sec:ProblemSetup}

Given a discrete-time plant $P$ described by
\begin{eqnarray}
\label{eq:Plant}
\nonumber
x(t+1) & = &  f(x(t),u(t)) \\ 
y(t) & = & g(x(t)) \\
\nonumber
v(t) & = & h(x(t))
\end{eqnarray}
where $t \in \mathbb{Z}_+$,
$x(t) \in \mathbb{R}^n$,
$u(t) \in \mathcal{U}$, $y(t) \in \mathcal{Y}$, $v(t) \in \mathcal{V}$,
and functions $f: \mathbb{R}^n \times \mathcal{U} \rightarrow \mathbb{R}^n$,
$g: \mathbb{R}^n \rightarrow \mathcal{Y}$
and $h: \mathbb{R}^n \rightarrow \mathcal{V}$ are given.
No apriori constraints are placed on the alphabet set $\mathcal{V}$:
It may be a Euclidean space, the set of reals, or a countable or finite set.
$\mathcal{U}$ and $\mathcal{Y}$ are
given finite alphabets with $| \mathcal{U}| = m$ and $|\mathcal{Y}|=p$, respectively:
They may represent quantized values of some analog inputs and outputs, 
or they may simply be symbolic inputs and outputs in general.  
We are also given a performance objective
\begin{equation}
\tag{\ref{eq:ObjectiveP}}
\sup_{T \geq 0} \sum_{t=0}^{T} \mu(v(t)) < \infty.
\end{equation}
Our goals are twofold:
\begin{enumerate}
\item To provide a systematic methodology for constructing a $\rho/\mu$ approximation of $P$.
\item To rigorously analyze the relevant properties of this construct.
\end{enumerate}

\section{A Special Structure }
\label{Sec:SpecialStructure}

In \cite{CHAPTER:TaMeDa2007},
we proposed a special `observer-inspired' structure and used it in conjunction with
a particular state-space based construct 
in order to approximate and subsequently design stabilizing controllers for a special class of systems,
namely switched second order homogenous systems with binary outputs.
In what follows, we begin in Section \ref{SSec:ChoiceW} by proposing a slight generalization of this structure,
by modifying it to allow for arbitrary (i.e. not necessarily binary) finite sensor output alphabets.
We also address the related question of minimal construction of the disturbance alphabet set $\mathcal{W}$.
Next, we show in Section \ref{SSec:ExistenceOfMap} that under one additional assumption, 
this generalized structure ensures the existence of 
function $\psi_i$ as required in property a) of Definition \ref{Def:DFMApproximation}.

\subsection{Generalized Structure and Minimal Choice of $\mathcal{W}$}
\label{SSec:ChoiceW}

 \begin{figure*}[thpb]
       \centering
       \includegraphics[scale=0.4]{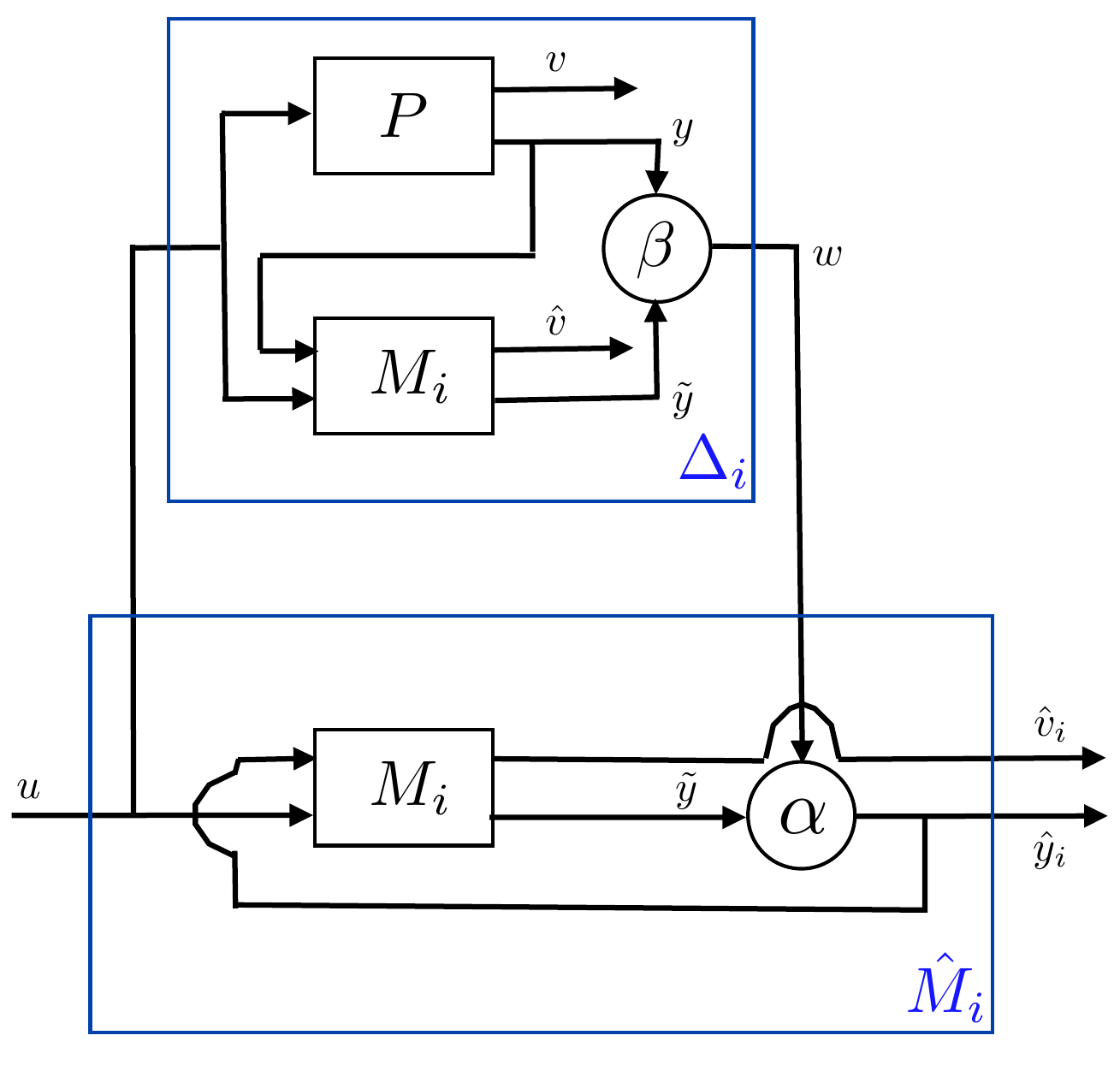}
       \caption{Proposed generalized structure for $\hat{M}_i$ and $\Delta_i$}
       \label{Fig:Structure}
   \end{figure*}

Consider the structure for $\hat{M}_i$ and $\Delta_i$ shown in Figure \ref{Fig:Structure},
where $M_i$ is a DFM.
To ensure that the interconnection is well-posed, 
we require $M_i$ to be $\mathcal{Y}/\mathcal{Y}$ strictly proper:
That is, its instantaneous output $\tilde{y}(t)$ is not an explicit function of its instantaneous input $y(t)$.

Noting that there is no loss of generality in 
assuming that a finite set $\mathcal{W}$ with cardinality $r+1$ is given by $\mathcal{W} = \{ 0,\hdots, r\}$,
we begin by showing that when $P$ is a system over finite alphabets, 
it is always possible to construct functions $\alpha$ and $\beta$ satisfying the property:
\begin{equation}
\label{Eq:AlphaBetaCondition}
\alpha \Big( \tilde{y}, \beta(\tilde{y},y) \Big)=y, \textrm{            for all   } y, \tilde{y} \in \mathcal{Y}.
\end{equation}
The relevance of this property will become clear in Section \ref{SSec:ExistenceOfMap}:
Intuitively,
$\beta$ and $\alpha$ play the role of subtraction and addition in the finite alphabet setting.

\smallskip
\begin{prop}
\label{Prop:AlphaBetaCondition}
Consider an alphabet set $\mathcal{Y}$ with $|\mathcal{Y}|=p$
and a set $\mathcal{W} = \{0, \hdots,r \}$.
For sufficiently large $r$, there always exists functions
$\beta: \mathcal{Y \times \mathcal{Y} \rightarrow \mathcal{W}}$ 
and 
$\alpha: \mathcal{Y} \times \mathcal{W} \rightarrow \mathcal{Y} \cup \{\epsilon\}$ 
such that (\ref{Eq:AlphaBetaCondition}) holds.
\end{prop} 

\begin{proof}
The proof is by construction.
Let $r= p^2-1$. 
Note that $|\mathcal{W}| = |\mathcal{Y}^2|=p^2$,
and there thus exists a bijective map $\beta : \mathcal{Y}^2 \rightarrow \mathcal{W}$
that associates with every pair
$(y_1,y_2) \in \mathcal{Y}^2$ a unique element of $\mathcal{W}$.
Now consider $\alpha: \mathcal{Y} \times \mathcal{W}  \rightarrow \mathcal{Y} \cup \{\epsilon\}$
defined by
\begin{displaymath}
\alpha(\tilde{y},w)= \left\{ \begin{array}{cc}
y & \textrm{          if    } \beta(\tilde{y},y)=w \\
\epsilon & \textrm{          otherwise}
\end{array} \right. .
\end{displaymath}
We have $ \alpha \Big( \tilde{y}, \beta(\tilde{y},y) \Big)=y $
for all $y, \tilde{y} \in \mathcal{Y}$, as desired.
\end{proof}
\smallskip

We next direct our attention in this setting to the choice of alphabet set $\mathcal{W}$.
A set with minimal cardinality is desirable, 
as the complexity of solving the full state feedback control synthesis problem grows with the cardinality of $\mathcal{W}$, 
as seen in the definition of $\mathbb{T}(J(q))$ in Theorem \ref{Thm:DPsynthesis}.
We thus answer the following question: 
What is the minimal cardinality of $\mathcal{W}$ for which one can construct 
functions $\beta$ and $\alpha$ with the desired property (\ref{Eq:AlphaBetaCondition})? 

\smallskip
\begin{lemma}
\label{Lemma:CardinalityW}
Given a set $\mathcal{Y}$ with $| \mathcal{Y }| =p$.
Let $\mathcal{W}^* = \{ 0, 1, \hdots, p^*-1\}$ be the smallest set for which there 
exists $\beta: \mathcal{Y} \times \mathcal{Y} \rightarrow \mathcal{W}$ 
and $\alpha: \mathcal{Y} \times \mathcal{W} \rightarrow \mathcal{Y} \cup \{ \epsilon \}$
satisfying $\alpha \Big( \tilde{y}, \beta(\tilde{y},y) \Big)=y$ for all $y,\tilde{y} \in \mathcal{Y}$.
We have $p^*=p$.
\end{lemma}

\begin{proof}
\begin{table}
\centering
  \begin{tabular}{ l || c  c c c c c}
 & $y_1$ & $y_2$ & $y_3$ & $y_4$ & $\hdots$ &$y_p$ \\
    \hline \hline 
$y_1$ & $0$ & $p-1$ & $p-2$ & $p-3$ & $\hdots$ & $1$ \\ 
$y_2$ & $1$ & $0$ & $p-1$ & $p-2$ & $\hdots$ &$2$ \\ 
$y_3$ & $2$ & $1$ & $0$ & $p-1$ & $\hdots$ & $3$ \\ 
$y_4$ & $3$ & $2$ & $1$ & $0$ & & \\ 
$\vdots$ & $\vdots$ & $\vdots$ & $\vdots$& & $\ddots$ & \\ 
$y_p$ & $p-1$ & $p-2$ & $p-3$ & & & $0$
  \end{tabular}
  \caption{Definition of $\beta: \mathcal{Y} \times \mathcal{Y} \rightarrow \mathcal{W}$ when $|\mathcal{Y}|=p$}
  \label{Table:BetaDefinition}
\end{table}

Let $p^*=p$, and consider a map $\beta: \mathcal{Y} \times \mathcal{Y} \rightarrow \mathcal{W}$
defined as shown in the Table \ref{Table:BetaDefinition}, 
to be read as $\beta(y_1,y_1)=0$, $\beta(y_2,y_1)=1$, $\beta(y_1,y_2)=p-1$
and so on.
Note that by construction, each element of $\mathcal{W}$ appears {\it exactly once} in every row of the table.
Now consider function $\alpha: \mathcal{Y} \times \mathcal{W} \rightarrow \mathcal{Y}$
defined by
\begin{displaymath}
\alpha(\tilde{y},w) = y \textrm{     where   } w=\beta(\tilde{y},y).
\end{displaymath}
$\alpha$ is a well-defined function,
and it is straightforward to show that 
$\alpha \Big( \tilde{y}, \beta(\tilde{y},y) \Big)=y$ for all $y,\tilde{y} \in \mathcal{Y}$.

Finally, note that when $p^* <p$, 
some element of $\mathcal{W}$ would have to appear twice in each row of the table.
Equivalently, for every $\tilde{y} \in \mathcal{Y}$, there exists $y_1 \neq y_2 \in \mathcal{Y}$
such that $\beta(\tilde{y},y_1) = \beta(\tilde{y},y_2)$.
Now suppose there exists a function $\alpha$ 
such that $\alpha \Big( \tilde{y}, \beta(\tilde{y},y) \Big)=y$ for all $y,\tilde{y} \in \mathcal{Y}$.
We then have
\begin{displaymath}
y_1 = \alpha(\tilde{y},\beta(\tilde{y},y_1)) = \alpha(\tilde{y},\beta(\tilde{y},y_2)) = y_2,
\end{displaymath} 
leading to a contradiction.  
\end{proof}
\smallskip

Note that $\alpha(\mathcal{Y} \times \mathcal{W})= \mathcal{Y}$ in the construction 
presented in the proof of Lemma \ref{Lemma:CardinalityW}.
We can thus drop $\{ \epsilon \}$ from the co-domain of $\alpha$.

\subsection{Ensuring existence of $\psi_i$}
\label{SSec:ExistenceOfMap}

We now turn out attention to proving that, under one additional assumption on $M_i$,
the structure proposed in Section \ref{SSec:ChoiceW}
and shown in Figure \ref{Fig:Structure} ensures that condition a) of Definition \ref{Def:DFMApproximation} is met:

\smallskip
\begin{lemma}
\label{Lemma:ConditionA}
Consider the system shown in Figure \ref{Fig:Structure},
where $P \subset \mathcal{U}^{\mathbf{Z}_+} \times (\mathcal{Y} \times \mathcal{V})^{\mathbf{Z}_+}$, 
$\mathcal{U}$ and $\mathcal{Y}$ are finite, 
and $\beta:\mathcal{Y} \times \mathcal{Y} \rightarrow \mathcal{W}$
and $\alpha: \mathcal{Y} \times \mathcal{W} \rightarrow \mathcal{Y}$ 
are given functions that satisfy (\ref{Eq:AlphaBetaCondition}). 
For any DFM
$M_i \subset (\mathcal{U} \times \mathcal{Y})^{\mathbb{Z}+} \times (\mathcal{Y \times \hat{\mathcal{V}}}_i)^{\mathbb{Z}_+} $ 
that is $\mathcal{Y} / \mathcal{Y}$ strictly proper and has fixed initial condition, 
there exists a $\psi_i : P \rightarrow \hat{P}_i$, 
where $\hat{P}_i$ is the interconnection of $\hat{M}_i$ and $\Delta_i$, 
such that $\psi_i$ is surjective and
$ \psi_i \Big( P|_{\mathbf{u_o},\mathbf{y_o}} \Big) \subseteq\hat{P}_i|_{\mathbf{u_o},\mathbf{y_o}}$. 
\end{lemma}

\begin{proof}
The proof is by construction.
We begin by noting that condition (\ref{Eq:AlphaBetaCondition})
ensures that the output $\mathbf{\hat{y}} \in \mathcal{Y}^{\mathbb{Z}_+}$ of $\hat{P}_i$ matches
the output $\mathbf{y} \in \mathcal{Y}^{\mathbb{Z}_+}$ of $P$ for every choice of 
$\mathbf{u} \in \mathcal{U}^{\mathbb{Z}_+}$.
Now consider $\psi_{1,i} : P \rightarrow M_i$ defined by
$\psi_{1,i}\Big( (\mathbf{u},(\mathbf{y},\mathbf{v}) \Big) = ((\mathbf{u},\mathbf{y}),(\tilde{\mathbf{y}},\hat{\mathbf{v}})) \in M_i$, 
where $(\tilde{\mathbf{y}},\hat{\mathbf{v}})$ is the unique output response of $M_i$ to input $(\mathbf{u},\mathbf{y})$ 
for fixed initial condition $q_o$. 
Also consider $\psi_{2,i} : \psi_{1,i} (P) \rightarrow \hat{P}_i$ defined by:
\begin{displaymath}
\psi_{2,i} \Big( ((\mathbf{u},\mathbf{y}),(\tilde{\mathbf{y}},\hat{\mathbf{v}})) \Big) = (\mathbf{u},(\mathbf{y},\hat{\mathbf{v}}))
\end{displaymath}
This map is well-defined and its image lies in $\hat{P}_i$ by virtue of the structure considered.
Let $\psi_i= \psi_{2,i} \circ \psi_{1,i}$. 
Note that $\psi_i$ is surjective since $\psi_{2,i}$ is surjective and $\psi_{2,i}^{-1}(\hat{P}_i) = \psi_{1}(P)$ 
by definition. 
Moreover, 
$\psi_i(P|_{\mathbf{u_o},\mathbf{y_o}}) \subseteq \hat{P}_i|_{\mathbf{u_o},\mathbf{y_o}}$ since 
\begin{displaymath}
\psi_i(P|_{\mathbf{u_o},\mathbf{y_o}}) 
= \psi_{2,i}\Big( \psi_{1,i}(P|_{\mathbf{u_o},\mathbf{y_o}}) \Big) 
=  \psi_{2,i}\Big( ((\mathbf{u_o},\mathbf{y_o}),(\mathbf{\tilde{y}}, \mathbf{\hat{v}})) \in M_i \Big) 
\subseteq \hat{P}_i|_{\mathbf{u_o},\mathbf{y_o}}
\end{displaymath}
which concludes our proof.
\qedsymbol
\end{proof}
\smallskip

It follows from Lemma \ref{Lemma:ConditionA} that by restricting ourselves to 
approximations $\{ \hat{M}_i \}_{i=1}^{\infty}$ with the structure shown in Figure \ref{Fig:Structure},
where $M_i$ (for each $i\in \mathbb{Z}_+$) 
is a $\mathcal{Y}/\mathcal{Y}$ strictly proper DFM with {\it fixed} initial condition,
but otherwise arbitrary structure,
property a) of Definition \ref{Def:DFMApproximation} is guaranteed by construction,
and we only need worry about constructing $\{ M_i \}_{i=1}^{\infty}$ to satisfy properties $(b)$ and $(c)$.

\section{Construction of $M_i$}
\label{Sec:Construction}

What remains is to construct a sequence of DFM $\{M_i\}_{i=1}^{\infty}$ that, 
when used in conjunction with the generalized structure proposed in Section \ref{SSec:ChoiceW} and 
shown in Figure \ref{Fig:Structure},
ensures that properties b) and c) of Definition \ref{Def:DFMApproximation} are satisfied.
We begin by giving the intuition behind this construction in Section \ref{SSec:ConstructionInspiration},
before presenting the details of the construction in Section \ref{SSec:ConstructionDetails}.

\subsection{Inspiration for the Construction}
\label{SSec:ConstructionInspiration}

The inspiration for the construction comes from linear systems theory.
Indeed, consider a discrete-time SISO LTI system $S$ described by
\begin{eqnarray*}
\label{Eq:LTIDynamics}
x(t+1) & = & A x(t) + B u(t) \\
y(t) & = & C x(t) + D u(t) 
\end{eqnarray*}
where $t \in \mathbb{Z}_+$, $x(t) \in \mathbb{R}^n$, $u(t) \in \mathbb{R}$, $y(t) \in \mathbb{R}$, 
$A$, $B$ and $C$ are given matrices of appropriate dimensions,
and $D$ is a given scalar.
Assume that the pair $(C,A)$ is observable and the pair $(A,B)$ is reachable.
Under these conditions, following a fairly classical derivation that is omitted here for brevity,
we can express the state of the system at the current
time in terms of its past $n$ inputs and outputs as
\begin{equation}
\label{Eq:RelatingStates}
x(t) = \left[ \begin{array}{cc} A^n O^{-1} & R - A^n O^{-1} M \end{array} \right] 
\left[ \begin{array}{c} 
y(t-1) \\ \vdots \\ y(t-n) \\ u(t-1) \\ \vdots \\ u(t-n)
\end{array} \right],
\end{equation}
where $R= [B \; AB \; \hdots \; A^{n-1}B]$ is the reachability matrix,
\begin{displaymath}
O= \left[ \begin{array}{c}
CA^{n-1} \\
CA^{n-2} \\
\vdots \\
C
\end{array} \right]
\end{displaymath}
is a row permutation of the observability matrix,
and $M$ is the matrix of Markov parameters
\begin{displaymath}
M= \left[ \begin{array}{cccc}
D & CB & \hdots  & CA^{n-2}B \\
0  & \ddots & \ddots & \vdots \\
\vdots & \ddots & \ddots & CB \\
0 & \hdots &  0 & D
\end{array} \right].
\end{displaymath}

This observation suggests an approach for constructing a sequence of approximate models 
of $S$ starting from finite length input-output sequence pairs of $S$:
The states of the $i^{th}$ approximate model, $\hat{S}_i$, 
are then those subsets of $\mathbb{R}^{2i}$ that constitute 
{\it feasible} snapshots of length $i$ of the input-output behavior of $S$.
Equivalently, each state of $\hat{S}_i$ corresponds to a subset 
of states of $S$, consisting of those states that are un-falsified by the observed data
of length $i$. 

In particular, when $i= n$, consider the approximate model $\hat{S}_n$ with state $\hat{x}(t)$ defined as 
\begin{displaymath}
\hat{x}(t) = [y(t-1), \hdots, y(t-n), u(t-1),\hdots,u(t-n)]'
\end{displaymath}
and state-space description
\begin{eqnarray*}
\hat{x}(t+1)  & = & \hat{A} \hat{x}(t) + \hat{B} u(t) \\
\hat{y}(t) & = & \hat{C} \hat{x}(t) + D u(t)
\end{eqnarray*}
where $\hat{C} = \left[ \begin{array}{cc} CA^n O^{-1} & CR-CA^n O^{-1} M \end{array} \right]$
and $\hat{A}$ and $\hat{B}$ are appropriately defined\footnote{The exact expression for $\hat{A}$ and 
$\hat{B}$ is not relevant to the discussion, 
and is thus omitted for brevity.} matrices.
We note the following:
\begin{enumerate}

\item If systems $S$ and $\hat{S}_n$ are identically initialized, 
meaning that their initial states obey
\begin{displaymath}
x(0) = \left[ \begin{array}{cc} A^n O^{-1} & R - A^n O^{-1} M \end{array} \right]  \hat{x}(0)
\end{displaymath}
their outputs will be identical for any choice of input $\mathbf{u} \in \mathbb{R}^{\mathbb{Z}_+}$.
In that sense, $\hat{S}_n$ can be considered to recover the original system $S$.

\item Every state of $\hat{S}_n$ corresponds to a single state of $S$.
The converse is not true. 
Indeed, there does not exist a one-to-one correspondence between the states of $S$ and $\hat{S}_n$:
The kernel of matrix $\left[ \begin{array}{cc} A^n O^{-1} & R - A^n O^{-1} M \end{array} \right]$ 
in (\ref{Eq:RelatingStates}) has non-zero dimension,
and one state of $S$ can correspond to several states of $\hat{S}_n$.
$\hat{S}_n$ is thus an inherently redundant model.

\end{enumerate}

An alternative approach for comparing the responses of $S$ and $\hat{S}_n$ 
without explicitly matching their initial states
is by considering an ``approximation error" $\Delta_i$ with the structure shown in Figure \ref{Fig:Structure}
($P$ then corresponds to ``$S$" and $M_i$ corresponds to ``$\hat{S}_n$").
In this setup, $\hat{S}_n$ is additionally given access to the outputs of $S$,
allowing it to estimate its initial state:
State $\hat{x}(t)$ of $\hat{S}_n$ can thus be thought of as its the best instantaneous estimate of the
state $x(t)$ of $S$. 
At time steps $t \leq n-1$, the state set of $\hat{S}$ is refined as follows
\begin{displaymath}
\hat{x}(0) \in \mathbb{R}^{2n}, 
\end{displaymath}
\begin{displaymath}
\hat{x}(1) \in \{v \in \mathbb{R}^{2n} |v(1)=y(0),  v_{n+1} = u(0) \}, 
\end{displaymath}
\begin{displaymath}
\hat{x}(2) \in \{ v \in \mathbb{R}^{2n} |v(1)=y(1), v(2)=y(0),v_{n+1} = u(1), v_{n+2} = u(0) \} 
\end{displaymath}
and so on. 
At time steps $t \geq n$, $\hat{x}(t)$ is uniquely defined by the expression in (\ref{Eq:RelatingStates}).
The $\mathcal{L}_2$ gain of $\Delta_i$,
defined here as the infimum of $\gamma \geq 0$ such that the inequality
\begin{displaymath}
\inf_{T \geq 0} \sum_{t=0}^{T} \gamma^2 \| u(t) \||^2 - \| w(t) \|^2 > -\infty
\end{displaymath}
holds, compares how well the outputs match after a transient
(i.e. after $\hat{S}_n$ is done estimating the initial state of $S$):
Since the outputs of $S$ and $\hat{S}_n$ will exactly match for all times $t \geq n$,
the $\mathcal{L}_2$ gain of $\Delta_i$ in this case is zero.

The internal structure of $\Delta_i$
thus has a nice intuitive interpretation that may not have been
as transparent to the readers when we introduced it in \cite{CHAPTER:TaMeDa2007},
and the problem of finite state approximation is thus intricately connected to that of 
state estimation and reconstruction under finite memory constraints.
Note that output $\hat{y}(t)$ cannot explicitly depend on input $y(t)$ in this setting, 
otherwise $\hat{S}_n$ can trivially match the output of $S$ at every time step, 
rendering the comparison meaningless.

While the use of $\hat{S}_n$ as an alternative model of $S$  
is not justifiable here,
this exercise suggests a procedure for constructing
approximations of systems over finite alphabets:
In that setting, $\mathcal{U}$ and $\mathcal{Y}$ are finite
leading to approximate models with {\it finite} state-spaces.

\subsection{Details of the Construction}
\label{SSec:ConstructionDetails}

Given a plant over finite alphabets as in (\ref{eq:Plant}) and a performance objective 
as in (\ref{eq:ObjectiveP}), we construct the corresponding sequence $\{M_i\}_{i=1}^{\infty}$ as follows:
For each $i \in \mathbb{Z}_+$, $M_i$ is a $\mathcal{Y} / \mathcal{Y}$ strictly proper DFM described by
\begin{eqnarray}
\nonumber
q(t+1) & = & f_i(q(t),u(t),y(t)) \\
\tilde{y}(t) & = & g_i(q(t)) \\
\nonumber 
\hat{v}_i(t) & = & h_i(q(t)) 
\end{eqnarray}
where $t \in \mathbb{Z}_+$,
$q(t) \in \mathcal{Q}_i$, $u(t) \in \mathcal{U}$,
$y(t) \in \mathcal{Y}$, $\tilde{y}(t) \in \mathcal{Y}$,
and $\hat{v}_i(t) \in \hat{\mathcal{V}}_i$.

\bigskip
\noindent {\bf State Set:} The state set is
\begin{displaymath}
\mathcal{Q}_i = \mathcal{Q}_{i, F} \cup \mathcal{Q}_{i, I} \cup \{q_{\emptyset}, q_{o}\}
\end{displaymath}
where 

\noindent $\mathcal{Q}_{i, F}$ - Set of final states. This is where the state of $M_i$ evolves
for $t \geq i$. 

\noindent $\mathcal{Q}_{i, I}$ - Set of initial states. This is where the state of $M_i$ evolves
for $1 \leq t < i$. 

\noindent $q_{\emptyset}$ - Impossible state. This is where the state of $M_i$ transitions to
when it encounters an input-output pair that does not correspond to plant $P$.

\noindent $q_{o}$ - Initial state. This is the fixed initial state of $M_i$ at $t=0$.

More precisely, using the shorthand notation $f_{u}(x)$ to denote $f(x,u)$, we have

\bigskip
\noindent $\rhd$
$\mathcal{Q}_{i,F} \subset \mathcal{Y}^i \times \mathcal{U}^i$, 
$q=(y_1,\hdots,y_i,u_1,\hdots,u_i) \in \mathcal{Q}_{i,F}$ if $\exists x_o \in \mathbb{R}^n$ such that
\begin{eqnarray}
\nonumber
y_i & = & g\Big( x_o \Big) \\ \nonumber
y_{i-1} & = & g\Big( f_{u_i}(x_o) \Big) \\ \label{Eq:FeasibleF}
y_{i-2} & = & g \Big( f_{u_{i-1}} \circ f_{u_i}(x_o)\Big) \\ \nonumber
\vdots & = & \vdots \\ \nonumber
y_1 & = & g \Big( f_{u_2} \circ \cdots \circ f_{u_i}(x_o) \Big) 
\end{eqnarray}

\noindent $\rhd$
$\mathcal{Q}_{i,I} = \mathcal{Q}_{i,I,1} \cup \hdots \cup \mathcal{Q}_{i,I,i}$
where 
$\mathcal{Q}_{i,I,j} \subset \mathcal{Y}^j \times \mathcal{U}^j$
and
$q=(y_1,\hdots,y_j,u_1,\hdots,u_j) \in \mathcal{Q}_{i,I,j}$ if $\exists x_o \in \mathbb{R}^n$ such that
\begin{eqnarray}
\nonumber
y_j & = & g\Big( x_o \Big) \\ \nonumber
y_{j-1} & = & g\Big( f_{u_j}(x_o) \Big) \\ \label{Eq:FeasibleI}
\vdots & = & \vdots \\ \nonumber
y_1 & = & g \Big( f_{u_2} \circ \cdots \circ f_{u_j}(x_o) \Big) 
\end{eqnarray}

\noindent 

\bigskip
\noindent {\bf Transition Function:} The transition function 
$f_i: \mathcal{Q}_i \times \mathcal{U} \times \mathcal{Y} \rightarrow \mathcal{Q}_i$ 
is defined as follows:

\bigskip
\noindent $\rhd$ For $q=(y_1,\hdots,y_i,u_1,\hdots,u_i) \in \mathcal{Q}_{i,F}$, we define
\begin{displaymath}
f_i(q,u,y)= 
\left\{ \begin{array}{cc}
\overline{q} = (y,y_1,\hdots,y_{i-1}, u, u_1,\hdots, u_{i-1}) & \textrm{      if   } \overline{q} \in \mathcal{Q}_{i,F} \\
q_{\emptyset} & \textrm{      otherwise}
\end{array} \right.
\end{displaymath}

\noindent $\rhd$ For $q=q_o$, we define
\begin{displaymath}
f_i(q_o,u,y)= 
\left\{ \begin{array}{cc}
\overline{q} = (y, u) & \textrm{      if   } \overline{q} \in \mathcal{Q}_{i,I,1} \\
q_{\emptyset} & \textrm{      otherwise}
\end{array} \right.
\end{displaymath}

\noindent $\rhd$ For $q=(y_1,\hdots,y_j,u_1,\hdots,u_j) \in \mathcal{Q}_{i,I,j}$, we define
\begin{displaymath}
f_i(q,u,y)= 
\left\{ \begin{array}{cc}
\overline{q} = (y,y_1,\hdots,y_j, u, u_1,\hdots, u_j) & \textrm{      if   } 
\overline{q} \in \mathcal{Q}_{i,I,j+1} \cup \mathcal{Q}_{i,F} \\
q_{\emptyset} & \textrm{      otherwise}
\end{array} \right.
\end{displaymath}

\noindent $\rhd$ For $q_{\emptyset}$, we define $f_i(q_{\emptyset},u,y)=q_{\emptyset}$
for all $u \in \mathcal{U}$ and $y \in \mathcal{Y}$.

\bigskip
\noindent {\bf Output Functions:} We begin by associating with every $q \in \mathcal{Q}_i$
a subset $X(q)$ of $\mathbb{R}^n$ defined as follows:

\bigskip
\noindent $\rhd$ For $q=(y_1,\hdots,y_i,u_1,\hdots,u_i) \in \mathcal{Q}_{i,F}$, let
\begin{equation}
X_o= \{\ x_o \in \mathbb{R}^n | x_o \textrm{     satisfies     } (\ref{Eq:FeasibleF}) \}
\end{equation}
and define
\begin{equation}
X(q) = f_{u_1} \circ \hdots \circ f_{u_i}(X_o)
\end{equation}

\noindent $\rhd$ For $q=(y_1,\hdots,y_j,u_1,\hdots,u_j) \in \mathcal{Q}_{i,I,j}$, let
\begin{equation}
X_o= \{\ x_o \in \mathbb{R}^n | x_o \textrm{     satisfies     } (\ref{Eq:FeasibleI}) \}
\end{equation}
and define
\begin{equation}
X(q) = f_{u_1} \circ \hdots \circ f_{u_j}(X_o)
\end{equation}

\noindent $\rhd$ Define
\begin{equation}
X(q) =
\left\{ \begin{array}{cc}
\mathbb{R}^n, & \; q=q_o \\
\emptyset, & \; q = q_{\emptyset}
\end{array} \right.
\end{equation}

We can also associate with every $q \in \mathcal{Q}_i$ a subset $Y(q)$ of $\mathcal{Y}$ defined as
\begin{equation}
\label{Eq:Ydefinition}
Y(q) = g(X(q))
\end{equation}

We are now ready to define the output function $g_i: \mathcal{Q}_i \rightarrow \mathcal{Y}$ as
\begin{equation}
g_i(q) = 
\left\{ \begin{array}{ll}
y \textrm{   for some   } y \in \mathcal{Y}, & \: \: \: \textrm{if   } Y(q) =\emptyset \\
y \textrm{   for some   } y \in Y(q), & \: \: \: \textrm{otherwise} \
\end{array} \right.
\end{equation}

The output function $h_i: \mathcal{Q}_i \rightarrow \hat{\mathcal{V}}_i$ is defined as
\begin{equation}
\label{Eq:vOutputFunction}
h_i(q) = 
\left\{ \begin{array}{ll}
h \Big( \displaystyle \argmax_{x \in X(q)} \mu(h(x)) \Big), 
& \:\:\: q \in \mathcal{Q}_{i,F} \cup \mathcal{Q}_{i,I} \cup \{q_o\} \\
\displaystyle h \Big( \argmin_{x \in \mathbb{R}^n} \mu(h(x)) \Big), 
& \:\:\: q=q_{\emptyset}
\end{array} \right.
\end{equation}

\bigskip 
\noindent {\bf Output Set:} The output set $\hat{\mathcal{V}}_i$ is defined as
\begin{displaymath}
\hat{\mathcal{V}}_i = \bigcup_{q \in \mathcal{Q}_i} h_i(q)
\end{displaymath}

\begin{remark}
We conclude this section with a few observations:
\begin{enumerate}
\item The output of $M_i$ corresponding to a state $q$ is chosen {\it arbitrarily} among the
feasible options. The possibility of error is accounted for in the gain $\gamma_i$ of $\Delta_i$.
\item Our definition of the performance output function $h_i$ assumes that 
the map $\mu: \mathbb{R} \rightarrow \mathbb{R}$ has a well-defined minimum and maximum.
This places some mild restrictions on the original problem.
\end{enumerate}
\end{remark}

\begin{remark}
When $|\mathcal{U}| =m$ and $|\mathcal{Y}|=p$, 
the cardinality of the state set $\mathcal{Q}_i$ of $M_i$ satisfies
\begin{displaymath}
m^i \leq | \mathcal{Q}_i | \leq m^i p^i.
\end{displaymath}
The bounds follow from the fact that every input sequence of length $i$ is feasible,
and for each input sequence, 
the corresponding number of feasible output sequences of length $i$ can range from 1 to $p^i$.
For each state $q_i \in \mathcal{Q}_i$,
there is at least 1 and at most $p \cdot m$ possible state transitions.
\end{remark}

\section{$\rho / \mu$ Approximation Properties of the Construction}
\label{Sec:Properties}

In this Section, we show that the construction of $\{M_i\}_{i=1}^{\infty}$ proposed
in Section \ref{SSec:ConstructionDetails} together with the generalized structure 
proposed and analyzed in Section \ref{Sec:SpecialStructure} indeed  allows us to meet 
the remaining two properties of Definition \ref{Def:DFMApproximation},
namely properties b) and c).

\subsection{Conditions on the Performance Objectives}
\label{SSec:PropertyB}

   \begin{figure*}[thpb]
       \centering
       \includegraphics[scale=0.35]{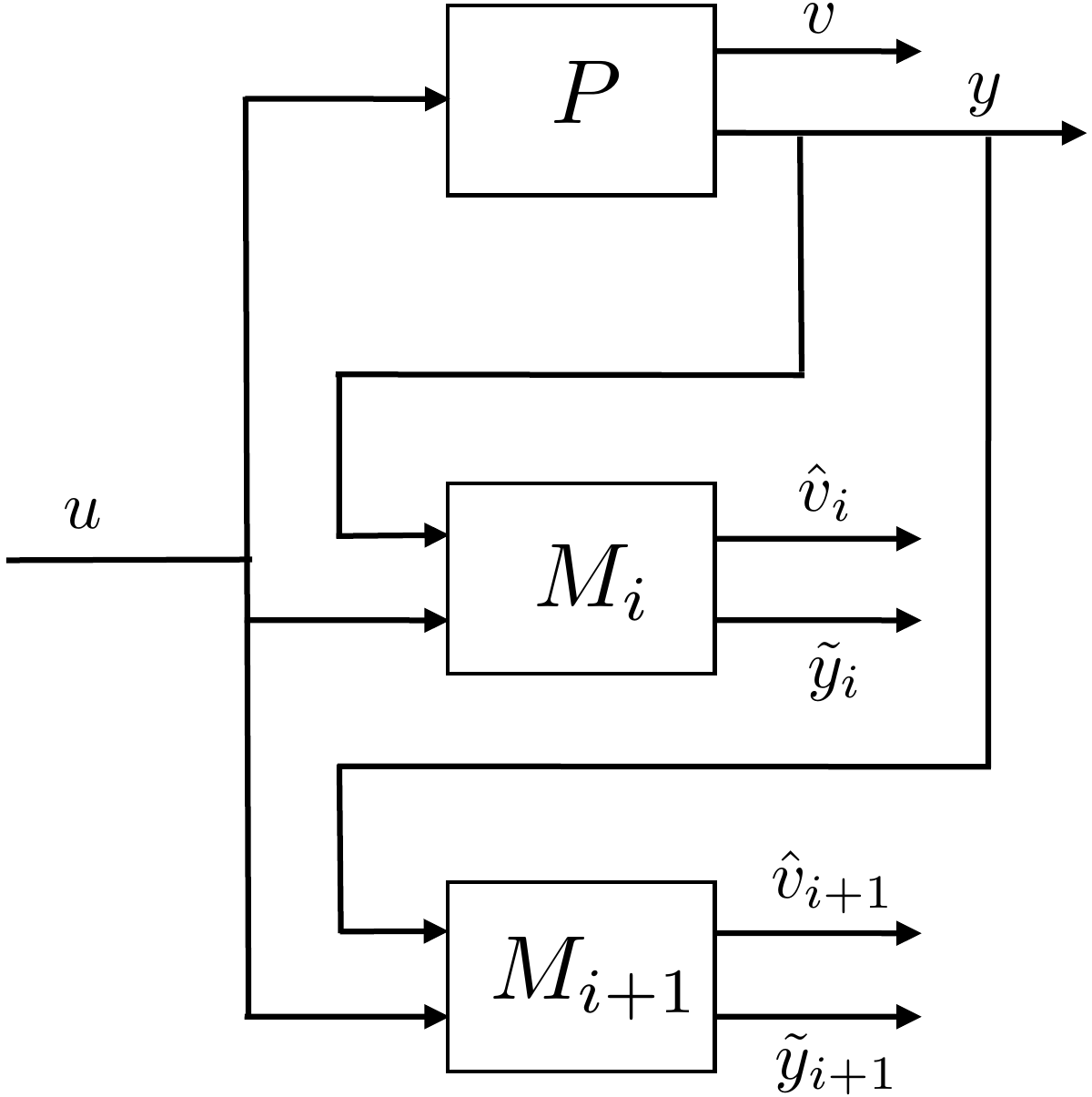}
       \caption{Interconnection of $P$, $M_i$ and $M_{i+1}$}
       \label{Fig:PMiInterconnection}
   \end{figure*}

\smallskip
\begin{prop}
\label{Prop:Xinclusion}
Consider a plant $P$ as in (\ref{eq:Plant}), 
a performance objective as in (\ref{eq:ObjectiveP}),
and a DFM $M_i$ constructed following the procedure given in 
Section \ref{SSec:ConstructionDetails} for some $i \geq 1$.
Consider the interconnection of $P$ and $M_i$ as shown in Figure \ref{Fig:PMiInterconnection}.
Let $x(t)$ and $q_i(t)$ be the states of $P$ and $M_i$,
respectively, at time $t$.
For any choice of $\mathbf{u} \in \mathcal{U}^{\mathbb{Z}_+}$ and $x(0) \in \mathbb{R}^n$,
we have
\begin{displaymath}
x(t) \in X(q_i(t)), \textrm{   for all   } t \geq 0.
\end{displaymath}
\end{prop}

\begin{proof}
Pick a choice $\mathbf{u} \in \mathcal{U}^{\mathbb{Z}_+}$ and $x(0) \in \mathbb{R}^n$.
At $t=0$, $q(0)=q_o$ and $X(q_o)=\mathbb{R}^n$ by construction.
Thus $x(0) \in X(q_i(0))$.
For $1 \leq t < i$, we can write
\begin{displaymath}
X(q_i(t)) = \Big\{ x \in \mathbb{R}^n  | x = f_{u(t-1)} \circ \hdots \circ f_{u(0)} (x_o) \textrm{        for some    } x_o \in \mathbb{R}^n \textrm{     that satisfies     } (\ref{Eq:FeasibleI}) \Big\}.
\end{displaymath}
Thus $x(t) \in X(q_i(t))$ since it can indeed be written in that form for some $x_o$, 
namely the initial state of $P$, $x_o=x(0)$,
and $x_o$ satisfies (\ref{Eq:FeasibleI}).
For $t \geq i$, we can write
\begin{displaymath}
X(q_i(t)) = \Big\{ x \in \mathbb{R}^n  | x = f_{u(t-1)} \circ \hdots \circ f_{u(t-i)} (x_o) \textrm{        for some    } x_o \in \mathbb{R}^n \textrm{     that satisfies     } (\ref{Eq:FeasibleF}) \Big\}.
\end{displaymath}
Again we have $x(t) \in X(q_i(t))$, since $x(t)$ can be written as $x(t) = f_{u(t-1)} \circ \hdots \circ f_{u(t-i)} (x(t-i))$,
and $x(t-i)$ satisfies (\ref{Eq:FeasibleF}).
Finally, we note that our argument is independent of the specific choice of $\mathbf{u} \in \mathcal{U}^{\mathbb{Z}_+}$,
and is also independent of the initial state of $P$,
which concludes our proof.
\end{proof}
\smallskip

\smallskip
\begin{prop}
\label{Prop:vInequality}
Consider a plant $P$ as in (\ref{eq:Plant}), 
a performance objective as in (\ref{eq:ObjectiveP}),
and a DFM $M_i$ constructed following the procedure given in 
Section \ref{SSec:ConstructionDetails} for some $i \geq 1$.
Consider the interconnection of $P$ and $M_i$ as shown in Figure \ref{Fig:PMiInterconnection}.
For any choice of $\mathbf{u} \in \mathcal{U}^{\mathbb{Z}_+}$ and $x(0) \in \mathbb{R}^n$,
we have
\begin{displaymath}
\mu(v(t)) \leq \mu(\hat{v}_i(t)), \textrm{   for all   } t \geq 0.
\end{displaymath}
\end{prop}

\begin{proof}
Pick a choice $\mathbf{u} \in \mathcal{U}^{\mathbb{Z}_+}$ and $x(0) \in \mathbb{R}^n$.
It follows from Proposition \ref{Prop:Xinclusion} that the corresponding state trajectories of
$P$ and $M_i$ satisfy $x(t) \in X(q_i(t))$, for all $t \in \mathbb{Z}_+$.
We have $q_i(t) \neq q_{\emptyset}$ for all $t$,
since $M_i$ is driven by a {\it feasible} pair $(\mathbf{u},\mathbf{y})$ of $P$
in this setup.
Let $\displaystyle \overline{x}_i(t) = \argmax_{x \in X(q_i(t))} \mu(h(x(t)))$.
It follows from (\ref{Eq:vOutputFunction}) that
\begin{equation*}
\mu(v(t)) = \mu(h(x(t)) \leq \mu \Big(h(x_i(t)) \Big) = \mu (\hat{v}_i (t)).
\end{equation*}
Once again, noting that our argument is independent of the specific choice of 
$\mathbf{u} \in \mathcal{U}^{\mathbb{Z}_+}$,
and of the initial state of $P$,
we conclude our proof.
\end{proof}
\smallskip

\smallskip
\begin{prop}
\label{Prop:vInequalityPlus}
Consider a plant $P$ as in (\ref{eq:Plant}), 
a performance objective as in (\ref{eq:ObjectiveP}),
two DFM $M_i$ and $M_{i+1}$ constructed following the procedure given in 
Section \ref{SSec:ConstructionDetails} for some $i \geq 1$.
Consider the interconnection of $P$, $M_i$ and $M_{i+1}$ as shown in Figure \ref{Fig:PMiInterconnection}.
For any choice of $\mathbf{u} \in \mathcal{U}^{\mathbb{Z}_+}$ and $x(0) \in \mathbb{R}^n$,
we have 
\begin{displaymath}
\mu(\hat{v}_{i+1}(t)) \leq \mu(\hat{v}_i(t)), \textrm{   for all   } t \geq 0.
\end{displaymath}
\end{prop}

\begin{proof}
Pick a choice $\mathbf{u} \in \mathcal{U}^{\mathbb{Z}_+}$ and $x(0) \in \mathbb{R}^n$.
Let $q_i(t)$ and $q_{i+1}(t)$ denote the states of $M_i$ and $M_{i+1}$,
respectively, at time $t$.
For $1 \leq t <i$, we can write
\begin{displaymath}
X(q_i(t)) = \Big\{ x \in \mathbb{R}^n  | x = f_{u(t-1)} \circ \hdots \circ f_{u(0)} (x_o) \textrm{        for some    } 
x_o \in \mathbb{R}^n \textrm{     that satisfies     } (\ref{Eq:FeasibleI}) \Big\}
\end{displaymath}
and
\begin{displaymath}
X(q_{i+1}(t)) = \Big\{ x \in \mathbb{R}^n  | x = f_{u(t-1)} \circ \hdots \circ f_{u(0)} (x_o) \textrm{        for some    } 
x_o \in \mathbb{R}^n \textrm{     that satisfies     } (\ref{Eq:FeasibleI}) \Big\}.
\end{displaymath}
Since $X(q_i(t)) = X(q_{i+1}(t))$, it follows from (\ref{Eq:vOutputFunction}) that
$\mu(\hat{v}_{i+1}(t)) = \mu(\hat{v}_i(t))$ for all $1 \leq t <i$.
For $t \geq i$, we can write
\begin{displaymath}
X(q_i(t)) = \Big\{ x \in \mathbb{R}^n  | x = f_{u(t-1)} \circ \hdots \circ f_{u(t-i)} (x_o) \textrm{        for some    } 
x_o \in \mathbb{R}^n \textrm{     that satisfies     } (\ref{Eq:FeasibleF}) \Big\}
\end{displaymath}
and
\begin{displaymath}
X(q_{i+1}(t)) = \Big\{ x \in \mathbb{R}^n  | x = f_{u(t-1)} \circ \hdots \circ f_{u(t-i-1)} (x_o) \textrm{        for some    } x_o \in \mathbb{R}^n \textrm{     that satisfies     } (\ref{Eq:FeasibleF}) \textrm{   with   } i+1 \textrm{   replacing   } i \Big\}.
\end{displaymath}
Thus $X(q_{i+1}(t)) \subseteq X(q_{i}(t))$.
Letting 
$$\displaystyle \overline{x}_i(t) = \argmax_{x \in X(q_i(t))} \mu(h(x(t)))$$
and 
$$\displaystyle \overline{x}_{i+1}(t) = \argmax_{x \in X(q_{i+1}(t))} \mu(h(x(t))),$$
it follows from (\ref{Eq:vOutputFunction}) that
\begin{equation*}
\mu(\hat{v}_{i+1}(t)) = \mu(h(\overline{x}_{i+1}(t))) \leq 
\mu(h(\overline{x}_i(t))) = \mu (\hat{v}_i (t)).
\end{equation*}
Finally, we note that our argument is independent of the specific choice of $\mathbf{u} \in \mathcal{U}^{\mathbb{Z}_+}$,
and is also independent of the initial state of $P$,
which concludes our proof.
\qedsymbol
\end{proof}
\smallskip

We can now state and prove the main result in this Section:

\smallskip
\begin{lemma}
\label{Lemma:ConditionB}
Consider a plant $P$ as in (\ref{eq:Plant}), 
a performance objective as in (\ref{eq:ObjectiveP}),
and a sequence of DFMs $\{M_i\}_{i=1}^{\infty}$ constructed following the procedure given in 
Section \ref{SSec:ConstructionDetails}
and used with the structure shown in Figure \ref{Fig:Structure}.
There exists a surjective map $\psi_i: P \rightarrow \hat{P}_i$
satisfying (\ref{Eq:MapCondition}) such that for every $(\mathbf{u},(\mathbf{y},\mathbf{v})) \in P$, we have
\begin{equation}
\tag{\ref{eq:Objectivenounds}}
\mu(v(t)) \leq \mu(\hat{v}_{i+1}(t)) \leq \mu(\hat{v}_i(t)),  
\end{equation}
for all $t \in \mathbb{Z}_+$, where 
$$ (\mathbf{u},(\mathbf{\hat{y}_i},\mathbf{\hat{v}_i})) = \psi_i \Big( ( \mathbf{u},(\mathbf{y},\mathbf{v}))  \Big),$$
$$(\mathbf{u},(\mathbf{\hat{y}_{i+1}},\mathbf{\hat{v}_{i+1}})) = \psi_{i+1} \Big( (\mathbf{u},(\mathbf{y},\mathbf{v}))  \Big).$$
\end{lemma}

\begin{proof}
Consider the map $\psi_i : P \rightarrow \hat{P}_i$ constructed in the proof of Lemma \ref{Lemma:ConditionA}.
We have $\psi_{i} = \psi_{2,i} \circ \psi_{1,i}$ where
 $\psi_{1,i} : P \rightarrow M_i$ is defined by
$$\psi_{1,i}\Big( (\mathbf{u_o},(\mathbf{y_o},\mathbf{v}) \Big) = ((\mathbf{u_o},\mathbf{y_o}),
(\tilde{\mathbf{y}}_i,\hat{\mathbf{v}}_i)) \in M_i.$$
Here $(\tilde{\mathbf{y}},\hat{\mathbf{v}}_i)$ is the 
{\it unique} output response of $M_i$ to input $(\mathbf{u}_o,\mathbf{y_o})$ 
for initial condition $q_i(0)$. 
Also recall that $\psi_{2,i} : \psi_{1,i} (P) \rightarrow \hat{P}_i$ was defined by
\begin{displaymath}
\psi_{2,i} \Big( ((\mathbf{u_o},\mathbf{y_o}),(\tilde{\mathbf{y}}_i,\hat{\mathbf{v}}_i)) \Big) 
= (\mathbf{u_o},(\mathbf{y_o},\hat{\mathbf{v}}_i)).
\end{displaymath}
Thus it suffices to show that for any $(\mathbf{u_o},(\mathbf{y_o},\mathbf{v})) \in P$,
the outputs of $M_i$ and $M_{i+1}$, 
$(\tilde{\mathbf{y}}_i,\hat{\mathbf{v}}_i)$ and $(\tilde{\mathbf{y}}_{i+1},\hat{\mathbf{v}}_{i+1})$,
respectively, in response to input
$(\mathbf{u}_o,\mathbf{y_o})$, satisfy the desired condition.
This follows directly from Propositions \ref{Prop:vInequality} and \ref{Prop:vInequalityPlus}.
\end{proof}
\smallskip

\subsection{Condition on the Gains}
\label{SSec:PropertyC}

In this Section, we first show that under some mild additional assumptions,
the proposed construction of $\{ M_i \}_{i=1}^{\infty}$ 
together with the structure shown in Figure \ref{Fig:Structure} meet the gain inequality in 
property c) of Definition \ref{Def:DFMApproximation}.
We begin by establishing some facts that will be useful in our analysis:

\smallskip
\begin{prop}
\label{Prop:Yinclusion}
Consider a plant $P$ as in (\ref{eq:Plant}), 
a performance objective as in (\ref{eq:ObjectiveP}),
and a DFM $M_i$ constructed following the procedure given in 
Section \ref{SSec:ConstructionDetails} for some $i \geq 1$.
Consider the interconnection of $P$ and $M_i$ as shown in Figure \ref{Fig:PMiInterconnection}.
Let $y(t)$ and $x(t)$ be the output and state, respectively, of $P$ at time $t$.
Let $q_i(t)$ be the state of $M_i$ at time $t$.
For any choice of $\mathbf{u} \in \mathcal{U}^{\mathbb{Z}_+}$ and $x(0) \in \mathbb{R}^n$,
we have
\begin{displaymath}
y(t) \in Y(q_i(t)), \textrm{   for all   } t \geq 0,
\end{displaymath}
for $Y$ defined in (\ref{Eq:Ydefinition}).
\end{prop}

\begin{proof}
Pick a choice $\mathbf{u} \in \mathcal{U}^{\mathbb{Z}_+}$ and $x(0) \in \mathbb{R}^n$.
By Proposition \ref{Prop:Xinclusion}, 
we have $x(t) \in X(q_i(t))$ for all $t \geq 0$.
It thus follows that $y(t) = g(x(t)) \in Y(q_i(t)) = g(X(q_i(t)))$, for all $t \geq 0$.
Finally, we note that our argument is independent of the specific choice of $\mathbf{u} \in \mathcal{U}^{\mathbb{Z}_+}$,
and is also independent of the initial state of $P$,
which concludes our proof.
\end{proof}
\smallskip

\smallskip
\begin{prop}
\label{Prop:YHierarchy}
Consider a plant $P$ as in (\ref{eq:Plant}), 
a performance objective as in (\ref{eq:ObjectiveP}),
and two DFMs $M_i$ and $M_{i+1}$ constructed following the procedure given in 
Section \ref{SSec:ConstructionDetails}, for some $i \geq 1$.
Consider the interconnection of $P$, $M_i$ and $M_{i+1}$ as shown in Figure \ref{Fig:PMiInterconnection}.
For any choice of $\mathbf{u} \in \mathcal{U^{\mathbb{Z}_+}}$ and initial state $x(0) \in \mathbb{R}^ n$
of $P$, we have
\begin{displaymath}
Y(q_{i+1}(t)) \subseteq Y(q_i(t)), \textrm{          for     } t \geq 0.
\end{displaymath}
\end{prop}

\begin{proof}
Pick a choice $\mathbf{u} \in \mathcal{U^{\mathbb{Z}_+}}$ and $x_o \in \mathbb{R}^n$.
By arguments similar to those made in the proof of Proposition \ref{Prop:vInequalityPlus},
omitted here for brevity,
we have
\begin{displaymath}
\left\{ \begin{array}{ll}
X(q_{i+1}(t)) = X(q_i(t)), \textrm{          for     } 0 \leq t < i \\
X(q_{i+1}(t)) \subseteq X(q_i(t)), \textrm{          for     } t \geq i
\end{array} \right.
\end{displaymath}
It thus follows, taking into account (\ref{Eq:Ydefinition}), that
\begin{displaymath}
\left\{ \begin{array}{ll}
Y(q_{i+1}(t)) = Y(q_i(t)), \textrm{          for     } 0 \leq t < i \\
Y(q_{i+1}(t)) \subseteq Y(q_i(t)), \textrm{          for     } t \geq i
\end{array} \right.
\end{displaymath}
which concludes our proof.
\end{proof}
\smallskip

\smallskip
\begin{defnt}
\label{Def:PositiveDefinite}
Let $\mathcal{W} = \{0, 1,\hdots, p-1\}$ for some integer $p$.
A function $\mu: \mathcal{W} \rightarrow \mathbb{R}_+$ is positive definite if
$\mu(w) \geq 0$ for all $w \in \mathcal{W}$ and $\mu(w)=0$ iff $w=0$.
\end{defnt}
\smallskip

\smallskip
\begin{defnt}
\label{Def:Flat}
Let $\mathcal{W} = \{0, 1,\hdots, p-1\}$ for some integer $p$ and consider a positive definite
function $\mu: \mathcal{W} \rightarrow \mathbb{R}_+$. 
$\mu$ is flat if there exists an $\alpha >0$ such that
$\mu(w)=\alpha$ for every $w \neq 0$.
\end{defnt}
\smallskip

\begin{defnt}
\label{Def:Child}
Consider a plant $P$ as in (\ref{eq:Plant}), 
a performance objective as in (\ref{eq:ObjectiveP}),
and a sequence of DFMs $\{ M_i \}_{i=1}^{\infty}$ constructed as described in Section \ref{SSec:ConstructionDetails}.
$q_{i+1} = (y_1,\hdots, y_j, y_{j+1},u_1,\hdots, u_j, u_{j+1}) \in \mathcal{Q}_{i+1}\setminus{\{q_o,q_{\emptyset}\}}$
is said to be a child of $q_i \in \mathcal{Q}_{i}$ if 
\begin{displaymath}
q _i= \left\{ \begin{array}{ll}
(y_1,\hdots,y_j,u_1,\hdots,u_j) & \textrm{      when   } j=i \\
q_{i+1} & \textrm{     when   } 1\leq j < i
\end{array} \right.
\end{displaymath}
We denote this by writing $q_{i+1} \in \mathcal{C}(q_i)$.
We consider $q_o$ and $q_{\emptyset}$ in $\mathcal{Q}_{i+1}$ to be children
of $q_o$ and $q_{\emptyset}$, respectively, in $\mathcal{Q}_{i}$.
\end{defnt}

\begin{prop}
Consider a plant $P$ as in (\ref{eq:Plant}), 
a performance objective as in (\ref{eq:ObjectiveP}),
and a sequence of DFMs $\{ M_i \}_{i=1}^{\infty}$ constructed as described in Section \ref{SSec:ConstructionDetails}.
For every $q_{i+1} \in \mathcal{Q}_{i+1}$,
there exists a unique $q_i \in \mathcal{Q}_i$ such that $q_{i+1} \in \mathcal{C}(q_i)$.
\end{prop}

\begin{proof}
Existence follows from Definition \ref{Def:Child} and the definition of the states. 
Uniqueness follows directly from Definition \ref{Def:Child}.
\end{proof}
\smallskip

\begin{remark}
The intuition here is that the set of states of $M_{i+1}$ can be partitioned into equivalence classes:
Elements of each equivalence class are children of the same state of $M_i$.
\end{remark}

\begin{prop}
\label{Prop}
Consider a plant $P$ as in (\ref{eq:Plant}), 
a performance objective as in (\ref{eq:ObjectiveP}),
and a sequence of DFMs $\{ M_i \}_{i=1}^{\infty}$ constructed as described in Section \ref{SSec:ConstructionDetails}.
For every $q_{i+1} \in \mathcal{Q}_{i+1}$,
$q_i \in \mathcal{Q}_i$ such that $q_{i+1} \in \mathcal{C}(q_i)$,
we have
$ X(q_{i+1}) \subseteq X(q_{i})$ and 
$ Y(q_{i+1}) \subseteq Y(q_{i})$.
\end{prop}

\begin{proof}
The proof follows directly from Definition \ref{Def:Child}
and the definitions of $X$ and $Y$.
\end{proof}
\smallskip

\smallskip
\begin{defnt}
\label{Def:OutputNested}
Consider a plant $P$ as in (\ref{eq:Plant}), 
a performance objective as in (\ref{eq:ObjectiveP}),
and a sequence of DFMs $\{ M_i \}_{i=1}^{\infty}$ constructed as described in Section \ref{SSec:ConstructionDetails}.
The sequence $\{M_i\}_{i=1}^{\infty}$ is output-nested if for every $i \in \mathbb{Z}_+$, 
$q_{i+1} \in \mathcal{Q}_{i+1}$ and $q_{i} \in \mathcal{Q}_{i}$ such that $q_{i+1} \in \mathcal{C}(q_{i})$,
if $g_{i}(q_i) \in Y(q_{i+1})$ then $g_{i+1}(q_{i+1})=g_{i}(q_i)$.
\end{defnt}
\smallskip

\begin{remark}
Intuitively, a sequence is output nested if every child is associated with the same output as its parent 
whenever that output is feasible for the child. 
\end{remark}

We can now prove the following:

\smallskip
\begin{prop}
\label{Prop:MuHierarchy}
Consider a plant $P$ as in (\ref{eq:Plant}), 
a performance objective as in (\ref{eq:ObjectiveP}),
and two DFMs $M_i$ and $M_{i+1}$ constructed following the procedure given in 
Section \ref{SSec:ConstructionDetails}, for some $i \geq 1$.
Consider the interconnection of $P$, $M_i$ and $M_{i+1}$ as shown in Figure \ref{Fig:PMiInterconnection}.
Let $w_i(t) = \beta (y(t),\tilde{y}_i(t))$ and 
$w_{i+1}(t) = \beta (y(t),\tilde{y}_{i+1}(t))$
for $\beta$ defined in Table \ref{Table:BetaDefinition},
and consider a flat, positive definite function $\mu_{\Delta}:\mathcal{W} \rightarrow \mathbb{R}_+$.
Assume that the sequence $\{ M_i \}_{i=1}^{\infty}$ is output nested.
For any choice of $\mathbf{u} \in \mathcal{U^{\mathbb{Z}_+}}$ and initial state $x(0) \in \mathbb{R}^ n$
of $P$, we have
\begin{displaymath}
\mu_{\Delta} (w_{i+1}(t)) \leq \mu_{\Delta} (w_i(t)),
\end{displaymath}
for all $t \geq 0$.
\end{prop}

\begin{proof}
Fix $i$.
Pick a choice $\mathbf{u} \in \mathcal{U^{\mathbb{Z}_+}}$, $x_o \in \mathbb{R}^n$.
Let $q_i(t)$ and $q_{i+1}(t)$ denote the states of $M_i$ and $M_{i+1}$,
respectively, at time $t$.
If $g_i(q_i(t)) \in Y(q_{i+1}(t))$,
we have $\tilde{y}_{i+1}(t) = g_{i+1} (q_{i+1}(t)) = g_{i} (q_i(t))=  \tilde{y}_i(t)$ 
since $\{ M_i \}_{i=1}^{\infty}$ is output nested.
Thus $w_{i+1}(t) = w_{i}(t)$,
and $\mu_{\Delta} (w_{i+1}(t)) = \mu_{\Delta} (w_i(t))$.
On the other hand,
if $g_i(q_i(t)) \notin Y(q_{i+1}(t))$,
we have $y(t) \neq \tilde{y}_i(t)$ since
$y(t) \in Y(q_{i+1}(t))$ by Proposition \ref{Prop:Yinclusion}.
It follows that $w_i(t) \neq 0$ and $\mu_{\Delta}(w_i(t)) =\alpha$, 
the unique positive number in the range of $\mu_{\Delta}$.
Meanwhile, $w_{i+1}(t)$ may or may not be zero,
and in both cases the inequality
$\mu_{\Delta} (w_{i+1}(t)) \leq \mu_{\Delta} (w_i(t))$
since $\mu_{\Delta}$ is flat and positive definite.

What is left is to note that our argument was independent of the choice of 
$\mathbf{u} \in \mathcal{U}^{\mathbb{Z}_+}$, $x(0) \in \mathbb{R}^n$, and $i$.
\end{proof}
\smallskip

\color{black}
We are now ready to state and prove the main result in this Section:

\smallskip
\begin{lemma}
\label{Lemma:ConditionC}
Consider a plant $P$ as in (\ref{eq:Plant}), 
a performance objective as in (\ref{eq:ObjectiveP}),
a disturbance alphabet $\mathcal{W} = \{0,\hdots,p-1 \}$ where $p=|\mathcal{Y}|$,
$\beta: \mathcal{Y} \times \mathcal{Y} \rightarrow \mathcal{W}$ defined as in Table \ref{Table:BetaDefinition},
a flat, positive definite function $\mu_{\Delta}: \mathcal{W} \rightarrow \mathbb{R}_+$,
and a sequence of DFM $\{M_i\}_{i=1}^{\infty}$ constructed following the procedure given in 
Section \ref{SSec:ConstructionDetails}.
Assume that $\{M_i\}_{i=1}^{\infty}$ is output nested.
For any $i \geq 1$,
the gains of $\Delta_i$ and $\Delta_{i+1}$ 
satisfy $\gamma_{i} \geq \gamma_{i+1}$.
\end{lemma}

\begin{proof}
Fix $i$, and let $\gamma_i$ be the gain of $\Delta_i$.
Pick a choice of $(\mathbf{u_o},(\mathbf{y_o},\mathbf{v})) \in P$,
and consider the setup shown in \ref{Fig:PMiInterconnection}.
Let $(\tilde{\mathbf{y}}_i,\hat{\mathbf{v}}_i)$ and $(\tilde{\mathbf{y}}_{i+1},\hat{\mathbf{v}}_{i+1})$
be the unique outputs of $M_i$ and $M_{i+1}$, respectively, 
in response to input $(\mathbf{u}_o,\mathbf{y_o})$.
Let $w_j(t) = \beta(\tilde{y}_j(t), y_o(t))$ for $j=i,i+1$.
It follows from Proposition \ref{Prop:MuHierarchy} that 
\begin{eqnarray*}
\mu_{\Delta}(w_{i+1}(t)) \leq \mu_{\Delta} (w_{i}(t)), \textrm{     } \forall t & \Leftrightarrow & 
- \mu_{\Delta}(w_{i+1}(t)) \geq -\mu_{\Delta} (w_{i}(t)), \textrm{      } \forall t \\
& \Leftrightarrow & 
\gamma_i \rho_{\Delta}(u_o(t)) - \mu_{\Delta}(w_{i+1}(t)) 
\geq \gamma_i \rho_{\Delta}(u_o(t)) -\mu_{\Delta} (w_{i}(t)),
\textrm{      } \forall t \\
& \Rightarrow & 
\sum_{t=0}^{T }\gamma_i \rho_{\Delta}(u_o(t)) - \mu_{\Delta}(w_{i+1}(t)) 
\geq 
\sum_{t=0}^{T} \gamma_i \rho_{\Delta}(u_o(t)) -\mu_{\Delta} (w_{i}(t)),
\textrm{      } \forall T \\
& \Rightarrow & 
\sum_{t=0}^{T }\gamma_i \rho_{\Delta}(u_o(t)) - \mu_{\Delta}(w_{i+1}(t)) 
\geq 
\inf_{t \geq 0} \sum_{t=0}^{T} \gamma_i \rho_{\Delta}(u_o(t)) -\mu_{\Delta} (w_{i}(t)),
\textrm{      } \forall T \\
& \Rightarrow & 
\inf_{T \geq 0} \sum_{t=0}^{T }\gamma_i \rho_{\Delta}(u_o(t)) - \mu_{\Delta}(w_{i+1}(t)) 
\geq 
\inf_{t \geq 0} \sum_{t=0}^{T} \gamma_i \rho_{\Delta}(u_o(t)) -\mu_{\Delta} (w_{i}(t))
\end{eqnarray*}
Letting $\tilde{\gamma}_{i+1} = \inf \gamma$ such that
\begin{displaymath}
\inf_{T \geq 0} \sum_{t=0}^{T }\gamma \rho_{\Delta}(u_o(t)) - \mu_{\Delta}(w_{i+1}(t)) > -\infty,
\end{displaymath}
we have $\tilde{\gamma}_{i+1} \leq \gamma_i$.
Since this argument holds for any choice of  $(\mathbf{u_o},(\mathbf{y_o},\mathbf{v})) \in P$,
we have
\begin{displaymath}
\gamma_{i+1} = \inf \{ \tilde{\gamma}_{i+1} \} \leq \gamma_i,
\end{displaymath}
where the `inf' is understood to be taken over all possible choices of feasible signals of $P$.
\end{proof}
\smallskip

\color{black}
\subsection{Ensuring Finite Error Gain}
\label{SSec:FiniteGain}

Note that Lemma \ref{Lemma:ConditionC},
while effectively establishing a hierarchy of approximations,
does not address the question: When is $\gamma_{i}$ finite?
A straightforward way to guarantee that is to require $\rho_{\Delta}(z) >0$ for all $z$.
While this may be meaningful in a setup where we have no preference for specific choices of control inputs 
(since $z=u$ in our proposed structure),
this may be too restrictive in general,
particularly when we wish to retain the ability to penalize certain inputs.

In this Section, we first propose a tractable
approach for establishing an upper bound for the approximation error:
The idea is to verify instead that an appropriately constructed DFM satisfies a suitably defined gain condition.
We then use this approach as the basis for deriving a readily verifiable 
sufficient condition for the gain to be finite.

We begin by associating with each approximate model $M_i$ 
two new DFMs:

\smallskip
\begin{defnt}
\label{Def:DFMExtension}
Consider a plant $P$ as in (\ref{eq:Plant}), 
a performance objective as in (\ref{eq:ObjectiveP}),
a disturbance alphabet $\mathcal{W} = \{0,\hdots,p-1 \}$ where $p=|\mathcal{Y}|$,
$\beta: \mathcal{Y} \times \mathcal{Y} \rightarrow \mathcal{W}$ defined as in Table \ref{Table:BetaDefinition},
a positive definite function $\mu_{\Delta}: \mathcal{W} \rightarrow \mathbb{R}_+$,
and a DFM $M_i$ constructed following the procedure given in 
Section \ref{SSec:ConstructionDetails}, for some $i \geq 1$.
The e-extension of $M_i$, denoted by $M_i^{e}$, is a new DFM,
$M_i^e \subset (\mathcal{U} \times \mathcal{Y})^{\mathbb{Z}_+} \times (\mathcal{Y \times \hat{\mathcal{V}}}_i \times \mathbb{R}_+)^{\mathbb{Z}_+}$,
obtained from $M_i$ by 
introducing one additional output $e : \mathcal{Q}_i \rightarrow \mathbb{R}_+$
defined by
\begin{displaymath}
e(q_i) = \left\{ \begin{array}{ll}
0 & \textrm{          if     } q_i=q_{\emptyset}\\
\displaystyle \max_{y_1,y_2 \in Y(q_i)} \mu_{\Delta} (\beta(y_1,y_2)) & \textrm{          otherwise}  
\end{array} \right. .
\end{displaymath}
\end{defnt}
\smallskip

\begin{remark}
It follows in Definition \ref{Def:DFMExtension} that when $q_i \neq q_{\emptyset}$,
$e(q_i)=0 \Leftrightarrow |Y(q_i)|=0$.
\end{remark}

\smallskip
\begin{defnt}
\label{Def:ZeroReduction}
Consider a plant $P$ as in (\ref{eq:Plant}), 
a performance objective as in (\ref{eq:ObjectiveP}),
a DFM $M_i$ constructed following the procedure given in 
Section \ref{SSec:ConstructionDetails}, for some $i \geq 1$,
and a choice $\rho_{\Delta}: \mathcal{U} \rightarrow \mathbb{R}_+$.
Let $\overline{\mathcal{U}} = \{ u\in \mathcal{U} | \rho_{\Delta}(u)=0 \}$.
The 0-reduction of $M_i$, denoted by $M_i^{0}$, is a new DFM,
$M_i^0 \subset (\overline{\mathcal{U}} \times \mathcal{Y})^{\mathbb{Z}_+} \times (\mathcal{Y \times \hat{\mathcal{V}}}_i \times \mathbb{R}_+)^{\mathbb{Z}_+}$,
obtained from $M_i^e$, the e-extension of $M_i$,
 by restricting the first input of $M_i^e$ to $\overline{\mathcal{U}}$. 
\end{defnt}
\smallskip

\begin{remark}
It follows from Definition \ref{Def:ZeroReduction} that a state $q_i$ of $M_i$,
and thus also of $M_i^e$,
$q_i=(y_1,\hdots,y_j,u_1,\hdots,u_j)$ for some $j \in \{0,\hdots,i\}$,
is a state of $M_i^0$ iff $u_k \in \overline{\mathcal{U}}$ for {\bf all} $k \in \{1,\hdots,j\}$.
The number of states of $M_i^0$ can thus be significantly lower than that of $M_i$ and $M_i^e$.
Likewise, the number of state transitions can be significantly lower.
\end{remark}

We are now ready to present an approach for verifying an upper bound for $\gamma_i$:

\smallskip
\begin{lemma}
\label{Lemma:GammaUpperBound}
Consider a plant $P$ as in (\ref{eq:Plant}), 
a performance objective as in (\ref{eq:ObjectiveP}),
a disturbance alphabet $\mathcal{W} = \{0,\hdots,p-1 \}$ where $p=|\mathcal{Y}|$,
$\rho_{\Delta}: \mathcal{U} \rightarrow \mathbb{R}_+$,
positive definite $\mu_{\Delta}: \mathcal{W} \rightarrow \mathbb{R}_+$,
and a DFM $M_i$ constructed following the procedure given in 
Section \ref{SSec:ConstructionDetails}, for some $i \geq 1$.
Let $\gamma_i$ be the gain of the corresponding error system $\Delta_i$
shown in Figure \ref{Fig:Structure} with $\beta: \mathcal{Y} \times \mathcal{Y} \rightarrow \mathcal{W}$ 
defined as in Table \ref{Table:BetaDefinition}.
Let $\hat{\gamma}_{i}$ be the infimum of $\gamma$ such that the e-extension of $M_i$, $M_i^e$, satisfies
\begin{equation}
\label{Eq:GainBound}
\inf_{T \geq 0} \sum_{t=0}^{T} \gamma \rho_{\Delta} (u(t)) - e_i(t) > -\infty.
\end{equation}
We have $\gamma_i \leq \hat{\gamma}_i$.
\end{lemma}
 
\begin{proof}
Assume that $M_i^e$ satisfies (\ref{Eq:GainBound}).
Pick a choice of $(\mathbf{u_o},(\mathbf{y_o},\mathbf{v})) \in P$ and
consider the interconnection of $P$ and $M_i$ shown in Figure \ref{Fig:PMiInterconnection}.
Let $x(t)$ and $q_i(t)$ be the states of $P$ and $M_i$, respectively, at time $t$,
and let $e_i(t)$ be the output of $M_i^e$ for input $(\mathbf{u_o},\mathbf{y}_o)$.
Note that the state of $M_i^e$ at time $t$ is also $q_i(t)$. 

If $|Y(q_i(t))| =1$, we have $e_i(t)=0$ by definition.
It also follows from Proposition \ref{Prop:Yinclusion}
and the fact that $Y(q_i(t))$ is a singleton
that $y(t) = \tilde{y}_i(t)$,
and thus $w(t) =0$ by the definition of $\beta$.
We thus have $e_i(t) = \mu_{\Delta}(w(t))$.
When $|Y(q_i(t))| >1$, we have 
$\displaystyle e_i(t) = \max_{y_1,y_2 \in Y(q_i(t))} \mu_{\Delta}(\beta(y_1,y_2)) \geq \mu_{\Delta} (w(t))$,
where the inequality again follows from Proposition \ref{Prop:Yinclusion}.
It thus follows that $e_i(t) \geq \mu_{\Delta}(w(t))$ for all $t \geq 0$, and we can now write
\begin{eqnarray*}
\mu_{\Delta}(w(t)) \leq e_i(t), \textrm{     } \forall t & \Leftrightarrow & 
- \mu_{\Delta}(w(t)) \geq - e_i(t), \textrm{      } \forall t \\
& \Leftrightarrow & 
\hat{\gamma}_i \rho_{\Delta}(u_o(t)) - \mu_{\Delta}(w(t)) 
\geq \hat{\gamma}_i \rho_{\Delta}(u_o(t)) -e_i(t),
\textrm{      } \forall t \\
& \Rightarrow & 
\sum_{t=0}^{T } \hat{\gamma}_i \rho_{\Delta}(u_o(t)) - \mu_{\Delta}(w(t)) 
\geq 
\sum_{t=0}^{T} \hat{\gamma}_i \rho_{\Delta}(u_o(t)) - e_i(t),
\textrm{      } \forall T \\
& \Rightarrow & 
\sum_{t=0}^{T }\hat{\gamma}_i \rho_{\Delta}(u_o(t)) - \mu_{\Delta}(w(t)) 
\geq 
\inf_{t \geq 0} \sum_{t=0}^{T} \hat{\gamma}_i \rho_{\Delta}(u_o(t)) - e_i(t),
\textrm{      } \forall T \\
& \Rightarrow & 
\inf_{T \geq 0} \sum_{t=0}^{T }\hat{\gamma}_i \rho_{\Delta}(u_o(t)) - \mu_{\Delta}(w(t)) 
\geq 
\inf_{t \geq 0} \sum_{t=0}^{T} \hat{\gamma}_i \rho_{\Delta}(u_o(t)) -e_i(t)
\end{eqnarray*}
Letting $\tilde{\gamma}_{i} = \inf \gamma$ such that
\begin{displaymath}
\inf_{T \geq 0} \sum_{t=0}^{T }\gamma \rho_{\Delta}(u_o(t)) - \mu_{\Delta}(w(t)) > -\infty,
\end{displaymath}
we have $\tilde{\gamma}_{i} \leq \hat{\gamma}_i$.
Since this argument holds for any choice of  $(\mathbf{u_o},(\mathbf{y_o},\mathbf{v})) \in P$,
we have
\begin{displaymath}
\gamma_{i} = \inf \{ \tilde{\gamma}_{i} \} \leq \hat{\gamma}_i,
\end{displaymath}
where the `inf' is understood to be taken over all possible choices of feasible signals of $P$.
\end{proof}
\smallskip

Lemma \ref{Lemma:GammaUpperBound} essentially establishes an upper bound for the
gain $\gamma_i$ of $\Delta_i$, verified by checking that
$M_i^e$ satisfies a suitably defined gain condition.
Verifying that a DFM satisfies a gain condition can be 
systematically and efficiently done: 
Readers are referred to \cite{JOUR:TaMeDa2008} for the details.
Note that in practice, this approach is typically used for computing an upper bound 
to be used in lieu of the gain for control synthesis,
as the problem of computing the gain of $\Delta_i$ exactly is 
difficult, if not intractable, in general.

Note that to ensure that the gain $\gamma_i$ is finite, 
it suffices to ensure that its upper bound $\hat{\gamma_i}$ established using the approach in 
Lemma \ref{Lemma:GammaUpperBound} is finite.
We can take this a step further,
by proposing a more refined sufficient condition expressed in terms of the 0-reduction of $M_i$,
and that requires significantly less computational effort to verify:

\smallskip
\begin{lemma}
Consider a plant $P$ as in (\ref{eq:Plant}), 
a performance objective as in (\ref{eq:ObjectiveP}),
$\rho_{\Delta}: \mathcal{U} \rightarrow \mathbb{R}_+$,
a positive definite function $\mu_{\Delta}: \mathcal{W} \rightarrow \mathbb{R}_+$,
and a DFM $M_i$ constructed following the procedure given in 
Section \ref{SSec:ConstructionDetails}, for some $i \geq 1$.
Let $\gamma_i$ be the gain of the corresponding error system $\Delta_i$
shown in Figure \ref{Fig:Structure}.
Let $M_i^0$ be the 0-reduction of $M_i$.
If $M_i^0$ satisfies (\ref{Eq:GainBound})
for some finite $\gamma$, then $\gamma_i$ is finite.
\end{lemma}

\begin{proof}
Construct a weighted graph corresponding to $M_i^e$ by associating with every state transition of $M_i^e$ 
a cost, namely `$\gamma \rho_{\Delta}(u) -e_i$' defined by the input $u$ that drives the transition and
the output $e_i$ associated with the beginning state of the transition.
$M_i^e$ satisfies (\ref{Eq:GainBound}) iff every cycle in the corresponding weighted graph
has non-negative total cost -
the proof of this statement is omitted for brevity - readers are referred to \cite{JOUR:TaMeDa2008} for the details.
In particular, $\hat{\gamma}_i$, the infimum of $\gamma$ such that (\ref{Eq:GainBound}) is satisfied,
is infinite iff there exists a cycle in $M_i^e$, driven {\bf entirely} by inputs in $\overline{\mathcal{U}}$,
and such that $e_i \neq 0$ for at least one state along the cycle.
Thus, it suffices to verify that $M_i^0$ satisfies (\ref{Eq:GainBound}) for some finite $\gamma$ to ensure that
$\hat{\gamma_i} < \infty$, from which we can deduce that $\gamma_i$ is finite by Lemma \ref{Lemma:GammaUpperBound}.
\end{proof}
\smallskip

We conclude this section by proving that the gain {\it bounds} established in Lemma \ref{Lemma:GammaUpperBound}
satisfy the hierarchy required in condition c) of Definition \ref{Def:DFMApproximation}.

\smallskip
\begin{lemma}
\label{Lemma:ConditionCbound}
Consider a plant $P$ as in (\ref{eq:Plant}), 
a performance objective as in (\ref{eq:ObjectiveP}),
a disturbance alphabet $\mathcal{W} = \{0,\hdots,p-1 \}$ where $p=|\mathcal{Y}|$,
$\rho_{\Delta}: \mathcal{U} \rightarrow \mathbb{R}_+$,
positive definite $\mu_{\Delta}: \mathcal{W} \rightarrow \mathbb{R}_+$,
and a sequence of DFMs $\{M_i\}_{i=1}^{\infty}$ constructed following the procedure given in 
Section \ref{SSec:ConstructionDetails}.
Let $\hat{\gamma}_{i}$ be the infimum of $\gamma$ such that the e-extension of $M_i$, $M_i^e$, satisfies
(\ref{Eq:GainBound}).
We have $\hat{\gamma}_i \geq \hat{\gamma}_{i+1}$.
\end{lemma}
 
\begin{proof}
Fix $i$.
Pick a choice of $(\mathbf{u_o},(\mathbf{y_o},\mathbf{v})) \in P$ and
consider the interconnection of $P$, $M_i$ and $M_{i+1}$ as shown in Figure \ref{Fig:PMiInterconnection}.
Let $q_i(t)$ and $q_{i+1}(t)$ be the states of $M_{i}$ and $M_{i+1}$, respectively, at time $t$,
and let $e_i(t)$ and $e_{i+1}(t)$ be the outputs of the corresponding e-extensions $M_i^e$ 
and $M_{i+1}^e$, respectively, for input $(\mathbf{u_o},\mathbf{y}_o)$.
 
By Proposition \ref{Prop:YHierarchy},
we have $Y(q_{i+1}(t)) \subseteq Y(q_i(t))$, for all $t \geq 0$.
Thus we have for every $t \geq 0$:
\begin{displaymath}
e_{i+1}(t) =
\max_{y_1,y_2 \in Y(q_{i+1}(t))} \mu_{\Delta} (\beta(y_1,y_2)) 
\leq 
\max_{y_1,y_2 \in Y(q_{i}(t))} \mu_{\Delta} (\beta(y_1,y_2)) 
= e_{i}(t).
\end{displaymath}

We can now write for any $\gamma \geq 0$
\begin{eqnarray*}
-e_i(t) \leq -e_{i+1}(t), \textrm{      } \forall t 
& \Leftrightarrow& 
\gamma \rho_{\Delta}(u_o(t)) -  e_i(t)
\geq \gamma \rho_{\Delta}(u_o(t)) - e_{i+1}(t),
\textrm{      } \forall t \\
& \Rightarrow & 
\sum_{t=0}^{T } \hat{\gamma}_i \rho_{\Delta}(u_o(t)) - \mu_{\Delta}(w(t)) 
\geq 
\sum_{t=0}^{T} \hat{\gamma}_i \rho_{\Delta}(u_o(t)) - e(t),
\textrm{      } \forall T \\
& \Rightarrow & 
\sum_{t=0}^{T }\hat{\gamma}_i \rho_{\Delta}(u_o(t)) - \mu_{\Delta}(w(t)) 
\geq 
\inf_{t \geq 0} \sum_{t=0}^{T} \hat{\gamma}_i \rho_{\Delta}(u_o(t)) - e(t),
\textrm{      } \forall T \\
& \Rightarrow & 
\inf_{T \geq 0} \sum_{t=0}^{T }\hat{\gamma}_i \rho_{\Delta}(u_o(t)) - \mu_{\Delta}(w(t)) 
\geq 
\inf_{t \geq 0} \sum_{t=0}^{T} \hat{\gamma}_i \rho_{\Delta}(u_o(t)) -e(t)
\end{eqnarray*}
It thus follows that $\hat{\gamma}_i \geq \hat{\gamma}_{i+1}$.
\end{proof}
\smallskip

Note that Lemma \ref{Lemma:ConditionCbound} does not require the additional assumptions
(output nested $\{M\}_{i=1}^{\infty}$ and flat $\mu_{\Delta}$) that Lemma
\ref{Lemma:ConditionC} requires to hold.
That is because the gain bounds are inherently conservative, 
effectively considering a `worst case' scenario.

\section{Semi-Completeness of the Construct}
\label{Sec:Completeness}

In this Section, we prove one additional property of the given construct:
Intuitively, we show that if a deterministic finite state machine exists that can accurately predict
the sensor output of a plant after some initial transient,
then our construct recovers it. 
While the resulting
DFM generated by our construct is not expected to be minimal 
(due to the inherent redundancy in this description, see the discussion in Section
\ref{SSec:ConstructionInspiration}), 
this property suggests that our construct is well-suited for addressing analytical questions
about convergence of the approximate models to the original plant.

\smallskip
\begin{thm}
Consider a plant $P$ as in (\ref{eq:Plant}), 
a performance objective as in (\ref{eq:ObjectiveP}),
a positive definite choice of $\mu_{\Delta}: \mathcal{W} \rightarrow \mathbb{R}_+$,
and a sequence $\{ M_i \}_{i=1}^{\infty}$ constructed following the procedure given in 
Section \ref{SSec:ConstructionDetails},
with $\{ \gamma_i\}_{i=1}^{\infty}$ denoting the gains of the corresponding approximation errors $\{ \Delta_i \}_{i=1}^{\infty}$
shown in Figure \ref{Fig:Structure}.
Assume there exists a DFM $M$ with fixed initial condition,
such that the corresponding $\Delta$ obtained by interconnecting $P$ and $M$ as in Figure
\ref{Fig:Structure} has gain $\gamma=0$.
Then $\gamma_{i^*}=0$ for some index $i^*$.
Moreover, $\gamma_i =0$ for all $i \geq i^*$.
\end{thm}

\begin{proof}
Assume a DFM $M$ with the stated properties exists,
and let $w(t)$ be the output of the system $\Delta$ constructed by interconnecting 
$P$ and $M$ as shown in Figure \ref{Fig:Structure}.
By assumption, we have
\begin{displaymath}
\inf_{T \geq 0} \sum_{t=0}^{T} 0. \rho_{\Delta}(u(t)) - \mu_{\Delta}(w(t)) > -\infty
\Leftrightarrow \sup_{T \geq 0} \sum_{t=0}^{T} \mu_{\Delta}(w(t)) < \infty
\end{displaymath} 
Since $\mathcal{W}$ is finite, the cardinality of $\mu_{\Delta}(\mathcal{W})$ is also finite,
as is that of the state set of $M$.
Thus there must exists a time $T^*$ such that $\mu_{\Delta}(w(t))=0$ for all $t \geq T^*$, 
or equivalently  $w(t)=0$ for all $t \geq T^*$ (by the positive definiteness of $\mu_{\Delta}$).
Now let $i^*=T^*$,
and consider the corresponding DFM $M_{i^*}$ in the constructed sequence.
We claim that $|Y(q_{i^*})|=1$ for every $q_i^{*} \in \mathcal{Q}_{i^*,F}$.

The proof is by contradiction: Indeed,
suppose that $|Y(q_{i^*})| >1$ for some $q_{i^*}=(y_1,\hdots,y_{i^*},u_1,\hdots,u_{i^*})' \in \mathcal{Q}_{i^*,F}$.
Thus, there exists an input sequence, namely $u(0)=u_1$, $u(1)=u_2$,$\hdots$, $u(T^*-1)=u_{i^*}$
with two corresponding feasible sensor outputs of $P$ given by
$y(0)=y_1$, $y(1)=y_2$,$\hdots$, $y(T^*-1)=y_{i^*}$, $y(T^*)=y'$
and 
$y(0)=y_1$, $y(1)=y_2$,$\hdots$, $y(T^*-1)=y_{i^*}$, $y(T^*)=y''$
where $y' \neq y''$.
Since $M$ has fixed initial condition, its response to the input sequence is fixed,
and it thus follows that $w(T^*) \neq 0$ for some run, contradicting the fact that $w(t)=0$
for all $t \geq T^*$.
This cannot be, and hence $|Y(q_{i^*})| =1$ for all $q_{i^*} \in \mathcal{Q}_{i^*,F}$.

It follows from this and Proposition \ref{Prop:Xinclusion} that $y_{i^*}(t) =y(t)$ for every $t \geq T^*$,
and thus $\gamma_{i^*}=0$. 
Finally, when $i > i^*$, 
$|Y(q_{i})| = 1$ for every $q_{i} \in \mathcal{Q}_{i,F}$, and $\gamma_{i}=0$.
\end{proof}
\smallskip

\color{black}
\section{Conclusions and Future Work}
\label{Sec:Conclusions}

In this paper, 
we revisited the recently proposed notion of $\rho/\mu$ approximation and a corresponding 
particular structure for the approximate models and approximation errors.
We generalized this structure for the non-binary alphabet setting,
and we showed that the cardinality of the minimal disturbance alphabet 
that can be used in this setting equals that of the sensor output alphabet.
We then proposed a general, conceptual procedure for generating a sequence
of finite state machines for systems over finite alphabets that are not subject to exogenous inputs.
We explicitly derived conditions under which the resulting constructs,
used in conjunction with the generalized structure, 
satisfy the three required properties of $\rho/\mu$ approximations,
and we proposed a readily verifiable sufficient condition to ensure that the gain of the
approximation error is finite.
We also showed that these constructs exhibit a `semi-completeness' property,
in the sense that if a finite state machine exists that can
perfectly predict the sensor output after some transient,
then our construct recovers it.

Our future work will focus on two directions:
\begin{enumerate}
\item At the theoretical level, it is clear from the construct that the problem of 
approximation and that of state estimation under coarse sensing are closely intertwined.
We will thus focus on understanding the limitations of approximating certain classes of systems
using these constructs,
or at a more basic level,
the limitations of reconstructing the state under coarse sensing and finite memory constraints. 
\item At the algorithmic level, 
we will look into refining this procedure by developing a recursive version that
allocates available memory in a more selective manner,
in line with the dynamics of the system.
\end{enumerate}

\section{Acknowledgments}

This research was supported by NSF CAREER award 0954601 
and AFOSR Young Investigator award FA9550-11-1-0118.


\bibliographystyle{plain}
\bibliography{ReferencesNew}

\end{document}